\documentclass[12pt,reqno]{amsart}
\usepackage[a4paper,top=3cm,bottom=3cm,inner=3cm,outer=3cm]{geometry}

\usepackage{amsmath, amssymb, amsfonts, mathtools}
\usepackage{wasysym}

\usepackage{xcolor}
\definecolor{darkgreen}{rgb}{0,0.45,0}
\definecolor{darkred}{rgb}{0.6,0,0}
\usepackage[
    colorlinks, 
    citecolor=blue, 
    urlcolor=darkgreen, 
    final, 
    hyperindex, 
    pagebackref,
    linkcolor = darkred
]{hyperref}

\usepackage[capitalize]{cleveref}
\usepackage{xspace}
\usepackage{youngtab}

\usepackage{amsthm}

\theoremstyle{plain}
\newtheorem{thm}{Theorem}
\newtheorem*{thm*}{Theorem}
\newtheorem{prop}[thm]{Proposition}
\newtheorem{lem}[thm]{Lemma}
\newtheorem*{prop*}{Proposition}
\newtheorem{cor}[thm]{Corollary}

\theoremstyle{definition}
\newtheorem{defn}[thm]{Definition}

\theoremstyle{remark}

\title[Schur functors and categorified plethysm]{Schur functors and categorified plethysm}

\author[Baez]{John C.\ Baez$^{1,2}$} 
\author[Moeller]{Joe Moeller$^3$}
\author[Trimble]{\\Todd Trimble$^4$}

\address{$^1$Department of Mathematics, University of California, Riverside, CA 92521, USA}
\address{$^2$Centre for Quantum Technologies, National University of Singapore, 117543, Singapore}
\address{$^3$National Institute of Standards and Technology, Gaithersburg, MD 20899, USA}
\address{$^4$Department of Mathematics, Western Connecticut State University, Danbury, CT 06810, USA}

\email{baez@math.ucr.edu, joseph.moeller@nist.gov, trimblet@wcsu.edu}

\DeclareFontFamily{U}{min}{}
\DeclareFontShape{U}{min}{m}{n}{<-> udmj30}{}

\newcommand{\comp}{\bullet}

\newcommand{\define}[1]{{\bf \boldmath{#1}}\index{#1}}

\newcommand{\maps}{\colon}
\newcommand{\To}{\Rightarrow}

\newcommand{\ob}{\mathrm{Ob}}
\newcommand{\op}{^\mathrm{op}}

\newcommand{\core}{\mathsf{core}}
\newcommand{\disc}{\mathsf{disc}}
\renewcommand{\hom}{\mathsf{hom}}

\newcommand{\category}[1]{\mathsf{#1}}
\newcommand{\A}{\category A}
\newcommand{\B}{\category B}
\newcommand{\C}{\category C}
\newcommand{\D}{\category D}

\newcommand{\I}{\category I}
\newcommand{\J}{\category J}

\newcommand{\M}{\category M}
\newcommand{\N}{\mathbb N}

\newcommand{\R}{\category R}
\renewcommand{\S}{\category S}

\newcommand{\W}{\category W}
\newcommand{\X}{\category X}

\newcommand{\Z}{\mathbb Z}

\newcommand{\namedcat}[1]{\mathsf{#1}}

\newcommand{\Ab}{\namedcat{Ab}}
\newcommand{\Alg}{\namedcat{Alg}}
\newcommand{\Bialg}{\namedcat{Bialg}}

\newcommand{\LAdj}{\namedcat{LAdj}}
\newcommand{\RAdj}{\namedcat{RAdj}}

\newcommand{\Biring}{\namedcat{Biring}}
\newcommand{\Cat}{\namedcat{Cat}}

\newcommand{\CMon}{\namedcat{CMon}}

\newcommand{\G}{\mathsf{G}} 
\newcommand{\DG}{\mathsf{DG}} 
\newcommand{\Grp}{\namedcat{Grp}}
\newcommand{\Grpd}{\namedcat{Grpd}}
\newcommand{\ksbar}{\overline{k\S}}

\newcommand{\Poly}{\namedcat{Poly}}

\newcommand{\Rep}{\namedcat{Rep}}
\newcommand{\Rig}{\namedcat{Rig}}
\newcommand{\Ring}{\namedcat{Ring}}

\newcommand{\Set}{\namedcat{Set}}
\newcommand{\SMC}{\namedbicat{SMCat}}
\newcommand{\SMLin}{\namedbicat{SMLin}}
\newcommand{\Schur}{\namedcat{Schur}}
\newcommand{\TRig}{2\mhyphen\namedbicat{Rig}}
\newcommand{\TBirig}{2\mhyphen\namedbicat{Birig}}

\newcommand{\Vect}{\namedcat{Vect}}
\newcommand{\Fin}{\namedcat{Fin}}

\newcommand{\Mat}{\namedcat{Mat}}

\newcommand{\namedbicat}[1]{\mathbf{#1}}
\newcommand{\CCat}{\namedbicat{Cat}}

\newcommand{\Lin}{\namedbicat{Lin}}
\newcommand{\Cauch}{\namedbicat{Cauch}}
\newcommand{\GGrpd}{\namedbicat{Grpd}}

\newcommand{\bC}{\namedbicat{C}}

\renewcommand{\SS}{\mathbb{S}}

\mathchardef\mhyphen="2D
\newcommand{\ho}{_\mathsf{h}}

\usepackage{tikz, tikz-cd}
\usetikzlibrary{positioning}
\definecolor {processblue}{cmyk}{0.9,0.5,0,0}

\tikzstyle{simple}=[-,line width=2.000]
\tikzstyle{arrow}=[-,postaction={decorate},decoration={markings,mark=at position .5 with {\arrow{>}}},line width=1.100]
\pgfdeclarelayer{edgelayer}
\pgfdeclarelayer{nodelayer}
\pgfdeclarelayer{bg}
\pgfsetlayers{bg,edgelayer,nodelayer,main}
\tikzstyle{none}=[inner sep=-1pt]
\tikzstyle{species}=[circle,fill=none,draw=black,scale=1.0]
\tikzstyle{transition}=[rectangle,fill=none,draw=black,scale=1.15]
\tikzstyle{empty}=[circle,fill=none, draw=none]
\tikzstyle{inputdot}=[circle,fill=black,draw=black, scale=.5]
\tikzstyle{dot}=[circle,fill=black,draw=black]
\tikzstyle{bounding}=[circle,dashed, fill=none,draw=black, scale=9.00]
\tikzstyle{simple}=[-,draw=black,line width=1.000]
\tikzstyle{inarrow}=[-,draw=black,postaction={decorate},decoration={markings,mark=at position .5 with {\arrow{>}}},line width=1.000]
\tikzstyle{tick}=[-,draw=black,postaction={decorate},decoration={markings,mark=at position .5 with {\draw (0,-0.1) -- (0,0.1);}},line width=1.000]
\tikzstyle{inputarrow}=[->,draw=black, shorten >=.05cm]

\tikzset{main node/.style={circle,fill=blue!20,draw,minimum size=1cm,inner sep=0pt},}

\tikzstyle{construct}=[fill=white, draw=black, shape=circle]
\tikzstyle{universal}=[fill=black, draw=black, shape=circle]

\numberwithin{thm}{section}

\begin{document}

\begin{abstract}
It is known that the Grothendieck ring of the category of Schur functors---or equivalently, the representation ring of the permutation groupoid---is the ring of symmetric functions. This ring has a rich structure, much of which is encapsulated in the fact that it is a `plethory': a monoid in the category of birings with its substitution monoidal structure.  We show that similarly the category of Schur functors is a `2-plethory', which descends to give the plethory structure on symmetric functions. Thus, much of the structure of symmetric functions exists at a higher level in the category of Schur functors.   

\end{abstract}

\maketitle
\setcounter{tocdepth}{1} 
\tableofcontents

\section{Introduction}
\label{sec:intro}

The symmetric groups $S_n$ play a distinguished role in group theory, and their representations are rich in structure. Much of this structure only reveals itself if we collect all these groups into the groupoid $\S$ of finite sets and bijections. Let $\Fin\Vect$ be the category of finite-dimensional vector spaces over a field $k$ of characteristic zero. Any functor $\rho \maps \S \to \Fin\Vect$ can be expressed as a direct sum of representations of the groups $S_n$. We denote the full subcategory of $\Fin\Vect^\S$ consisting of finite direct sums of finite-dimensional representations by $\Schur$.

This category $\Schur$ binds all the finite-dimensional representations of all the symmetric groups $S_n$ into a single entity, revealing more structure than can be seen working with these groups one at a time. In particular, $\Schur$ has a monoidal structure called the `plethysm' tensor product, with respect to which $\Schur$ acts on the category $\Rep(G)$ of representations of \emph{any} group $G$. Each object of $\Schur$ acts as an endofunctor of $\Rep(G)$ called a  `Schur functor'. Thus, $\Schur$ plays a fundamental role in representation theory, which we aim to clarify.

We also apply our results on $\Schur$ to study the ring of symmetric functions, denoted $\Lambda$. This ring shows up in many guises throughout mathematics. For example:
\begin{itemize}
    \item It is the subring of $\Z[[x_1, x_2, \dots]]$ consisting of power series of bounded degree that are invariant under all permutations of the variables: these are called `symmetric functions'.
    \item It is the Grothendieck group of the category $\Schur$.
    \item It is the cohomology ring $H^*(BU)$, where $BU$ is the classifying space of the infinite-dimensional unitary group.
\end{itemize}
As a mere ring, $\Lambda$ is not very exciting: it is isomorphic to a polynomial ring in countably many generators. But $\Lambda$ is richly endowed with a plethora of further structure \cite{Macdonald, Egge}. Hazewinkel \cite{Hazewinkel} writes:
\begin{quote}
    It seems unlikely that there is any object in mathematics richer and/or more beautiful than this one [....]
\end{quote}

Following ideas of Tall and Wraith \cite{TallWraith}, Borger and Wieland \cite{Plethystic} defined a concept of `plethory' (which we call `ring-plethory') that encapsulates much of this rich structure on $\Lambda$. Here we derive the ring-plethory structure on $\Lambda$ from a `2-plethory' structure on $\Schur$, of which $\Lambda$ is the Grothendieck group. More than merely proving that $\Lambda$ is a ring-plethory, this shows that much of its rich structure exists \emph{at a higher level} in the category of Schur functors.

What is a ring-plethory? To understand this, it is good to start with the simplest example of all, $\Z[x]$, the ring of polynomials with integer coefficients in one variable. (Following the algebraic geometers, we always use `ring' to mean `commutative ring with unit'.) This ring $\Z[x]$ is the free ring on one generator. But besides the usual ring operations, $\Z[x]$ also has `co-operations' that act like ring operations going backwards. These are all derived by exploiting the freeness property. Namely, $\Z[x]$ is equipped with the unique ring homomorphisms that send $x$ to the indicated elements:
\begin{itemize}
    \item `coaddition':
    \begin{align*}
        \alpha \maps \Z[x] &\to \Z[x] \otimes \Z[x]
        \\
        x & \mapsto x \otimes 1 + 1 \otimes x
    \end{align*}
    \item `co-zero', or the `coadditive counit':
    \begin{align*}
        o \maps \Z[x] &\to \Z
        \\
        x & \mapsto 0
    \end{align*}
      \item `co-negation':
    \begin{align*}
        \nu \maps \Z[x] &\to
        \Z[x] 
        \\
        x & \mapsto - x
    \end{align*}
    \item `comultiplication':
    \begin{align*}
        \mu \maps \Z[x] &\to \Z[x] \otimes \Z[x]
        \\
        x & \mapsto x \otimes x
    \end{align*}
    \item and `co-one', or the `comultiplicative counit':
    \begin{align*}
        \epsilon \maps \Z[x] &\to \Z
        \\
        x & \mapsto 1.
    \end{align*}
\end{itemize}
These obey all the usual ring axioms if we regard them as morphisms in the \emph{opposite} of the category of rings. Thus we say $\Z[x]$ is a \define{biring}: a ring object in $\Ring\op$.  

Why is $\Z[x]$ a biring? Any ring can be seen as the ring of functions on some kind of space, and $\Z[x]$ is the ring of functions on the affine line, $\A^1$. Grothendieck made this into a tautology by \emph{defining} the category of affine schemes to be $\Ring\op$ and \emph{defining} $\A^1$ to be the ring $\Z[x]$ seen as an object in $\Ring\op$. But the affine line itself can be made into a ring, much like the real or complex line. Thus $\A^1$ becomes a ring object in the category of affine schemes. But this is precisely a biring! The formulas above express the ring operations on $\A^1$ as co-operations on $\Z[x]$.

A biring can equivalently be seen as a ring $B$ such that the representable functor 
\[
    \Ring(B, -) \maps \Ring \to \Set
\]
is equipped with a lift to a functor $\Phi_B$ taking values in the category of rings, as follows \cite{TallWraith}:
\[
\begin{tikzcd}[column sep = large]
    &
    \Ring
    \arrow[d, "U"]
    \\
    \Ring
    \arrow[ur, "\Phi_B"]
    \arrow[ur, swap, ""{name = phi}, phantom, bend right = 5, pos = 0.6]
    \arrow[r, swap, "{\Ring(B, -)}"]
    \arrow[from = phi, r, Rightarrow, "\sim", sloped]
    &
    \Set.
\end{tikzcd}
\] 
Given a biring $B$, the co-operations on $B$ give $\Ring(B, R)$ a ring structure for any ring $R$ in a way that depends functorially on $R$. For example, since $\Z[x]$ is the free ring on one generator, $\Ring(\Z[x], -)$ assigns to any ring its underlying set. Thus $\Ring(\Z[x], -)$ lifts to the identity functor on $\Ring$. This gives $\Z[x]$ a natural biring structure, and one can check that this is the one described above. 

This second viewpoint is fruitful because endofunctors on $\Ring$ can be composed. Though not all endofunctors on $\Ring$ are representable, those that are representable are closed under composition. Thus for any birings $B$ and $B'$ there is a biring $B' \odot B$ such that
\[
    \Phi_B \circ \Phi_{B'} \cong \Phi_{B' \odot B}. 
\]
This puts a monoidal structure $\odot$ on the category $\Biring$, called the composition tensor product. A \define{ring-plethory} is then a monoid object in $(\Biring, \odot)$.  Since the category $\Biring$ is defined as the \emph{opposite} of the category of ring objects in $\Ring\op$, $B$ is a ring-plethory when $\Phi_B$ is a comonad.

For example, since $\Ring(\Z[x], -)$ lifts to the identity functor on $\Ring$, $\Z[x]$ with the resulting biring structure is actually the unit object for the plethysm tensor product. The unit object in a monoidal category is always a monoid object in a canonical way, so $\Z[x]$ becomes a ring-plethory. Concretely, this ring-plethory structure on $\Z[x]$ simply captures the fact that one can \emph{compose} polynomials in one variable. 

A more interesting ring-plethory is $\Lambda$, the ring of symmetric functions.  Its structure is often described in terms of fairly elaborate algebraic constructions, even by those familiar with ring-plethories \cite{Plethystic,TallWraith}.  It seems not to be generally appreciated that there is a  conceptual explanation for all this structure. It is the purpose of this paper to provide that explanation. 

We achieve this by categorifying the story so far, developing a theory of 2-plethories, and showing that $\Schur$ is a 2-plethory. Using this fact we show that $\Lambda$, the Grothendieck group of $\Schur$, is a ring-plethory. 

Before doing this, we must categorify the concepts of ring and biring. Or rather, since it is problematic to categorify subtraction directly, we start by omitting additive inverses and work not with rings but with `rigs', which again we assume to be commutative. A \define{birig} is then a rig object in the opposite of the category of rigs. For example, the free rig on one generator is $\N[x]$, and this becomes a birig with co-operations defined just as for $\Z[x]$ above---except for co-negation.

The concept of plethory also generalizes straightforwardly from rings to rigs. In fact it generalizes to algebras of any monad $M$ on $\Set$. Such a generalized plethory has been called a `Tall--Wraith monoid' \cite{TallWraith, StaceyWhitehouse}, but we prefer to call it an `$M$-plethory' in order to refer to various specific monads $M$. If $M$ is the monad whose algebras are rings, then $M$-plethories are ring-plethories, but when $M$ is the monad for rigs, we call an $M$-plethory a `rig-plethory'. For example, just as $\Z[x]$ becomes a ring-plethory, $\N[x]$ becomes a rig-plethory. This captures the fact that we can compose polynomials in $\N[x]$.

There are various ways to categorify the concept of rig. Since our goal is to study Schur functors and some related classical topics in representation theory, we shall fix a field $k$ of characteristic zero and define a `2-rig' to be a symmetric monoidal Cauchy complete linear category. In more detail:
\begin{defn} 
\label{def:2rig}
    A \define{linear category} is an essentially small category enriched over $\Vect$, the category of vector spaces over $k$. A \define{linear functor} is a $\Vect$-enriched functor between linear categories. A linear category is \define{Cauchy complete} when it has biproducts and all idempotents split. A \define{symmetric monoidal linear category} is a linear category with a symmetric monoidal structure for which the tensor product is bilinear on hom-spaces. A \define{2-rig} is a  symmetric monoidal linear category that is also Cauchy complete. 
\end{defn}

In language perhaps more familiar to algebraists, a linear category is Cauchy complete when it has finite direct sums and any idempotent endomorphism has a cokernel. In the definition of 2-rig, we do not need to impose a rule saying that the tensor product preserves biproducts and splittings of idempotents in each argument since this is automatic: these are `absolute' colimits for linear categories, meaning they are preserved by any linear functor.  This understanding of absolute colimits for linear categories has been folklore at least since Lawvere first wrote about Cauchy completeness \cite{lawvere1973metric}, but details can be found in \cite[Cor.\ 4.22]{LackTendas}.

With our definition of 2-rig, $\Schur$ turns out to be the free 2-rig on one generator.  Many other important categories are also 2-rigs:
\begin{itemize}
    \item the category $\Fin\Vect$ of finite-dimensional vector spaces over $k$, 
    \item the category of representations of any group on finite-dimensional vector spaces, 
    \item the category of finite-dimensional $G$-graded vector spaces for any group $G$, 
    \item the category of bounded chain complexes of finite-dimensional vector spaces, 
    \item the category of finite-dimensional super vector spaces, 
    \item  for $k = \mathbb{R}$ or $\mathbb{C}$, the category of finite-dimensional vector bundles over any topological space, or smooth vector bundles over any smooth manifold, 
    \item  the category of algebraic vector bundles over any algebraic variety over $k$, 
    \item  the category of coherent sheaves of finite-dimensional vector spaces over any algebraic variety (or scheme or algebraic stack) over $k$.
\end{itemize}
Some of these categories are abelian, but categories of vector bundles are typically not. They are still Cauchy complete, and this is another reason we develop our theory at this level of generality.

There is a 2-category of 2-rigs, denoted $\TRig$, and we define a \define{2-birig} to be a 2-rig $\B$ such that the 2-functor $\TRig(\B, -) \maps \TRig \to \CCat$ is equipped with a lift to a 2-functor $\Phi_\B$ taking values in 2-rigs:
\[
\begin{tikzcd}[column sep = large]
    &
    \TRig
    \arrow[d, "U"]
    \\
    \TRig
    \arrow[ur, "\Phi_B"]
    \arrow[ur, swap, ""{name = phi}, phantom, bend right = 5, pos = 0.6]
    \arrow[r, swap, "{\TRig(B, -)}"]
    \arrow[from = phi, r, Rightarrow, "\sim", sloped]
    &
    \CCat
\end{tikzcd}
\]

There is also a 2-category of 2-birigs. Analogously with birings, for any 2-birigs $\B$ and $\B'$ there is a 2-birig $B \odot B'$ corresponding to endofunctor composition. This equips the 2-category $\TBirig$ with a monoidal structure. We define a \define{2-plethory} to be a pseudomonoid---roughly, a monoid object up to coherent isomorphism---in $(\TBirig, \odot)$. The multiplication in this pseudomonoid is called the `plethysm' tensor product. 

Just as $\N[x]$ is the free rig on one generator, we prove that $\Schur$ is the free 2-rig on one generator. It follows that $\TRig(\Schur, -)$ lifts to the identity functor on $\TRig$:
\[
\begin{tikzcd}[column sep = huge]
    &
    \TRig
    \arrow[d, "U"]
    \\
    \TRig
    \arrow[ur, "1"]
    \arrow[ur, swap, ""{name = phi}, phantom, bend right = 5, pos = 0.6]
    \arrow[r, swap, "{\TRig(\Schur, -)}"]
    \arrow[from = phi, r, Rightarrow, "\sim", sloped]
    &
    \CCat.
\end{tikzcd}
\] 
This makes $\Schur$ into a 2-birig, and since $1 \circ 1 = 1$, $\Schur$ becomes a 2-plethory. This captures the fact that we can compose Schur functors. 

Taking the Grothendieck group of $\Schur$, we obtain the known ring-plethory structure on $\Lambda$, the ring of symmetric functions. The birig structure is fairly straightforward. The rig-plethory structure takes considerably more work. Most subtle of all is the biring structure, and in particular the co-negation: this involves  $\Z_2$-graded chain complexes of Schur functors, and is connected to the `rule of signs' in Joyal's theory of species \cite{Especies}. 

\subsection*{Outline of the paper}

\cref{sec:Schur} begins with an overview of the classical theory of Schur functors.  We then introduce the category $\Poly$ of `polynomial species', which are a special kind of linear species in the sense of Joyal \cite{Especies}.  Then we give our abstract definition of Schur functors as endomorphisms of the forgetful 2-functor $U \maps \TRig \to \CCat$, and show that any polynomial species gives such a Schur functor.

\cref{sec:Poly} is dedicated to proving our first main result, \cref{thm:schur_versus_poly}, which says that the category $\Schur$ of abstract Schur functors is equivalent to the category $\Poly$. En route, we prove in Theorems  \ref{thm:poly_versus_kS} and \ref{thm:representing_U} that $\Poly$ is the underlying category of the free 2-rig on one generator, which we call $\ksbar$, and that this 2-rig represents the 2-functor $U \maps \TRig \to \CCat$.

In \cref{sec:2-birig}, we define 2-birigs, a categorification of the notion of biring. In \cref{thm:schur_2-birig} we show that $\Schur$ has a 2-birig structure coming from its equivalence with the free 2-rig on one generator.

In \cref{sec:plethysm}, we begin by exposing an alternative perspective on birigs. Birigs and birings are examples of the more general notion of `$M$-bialgebras': that is, bialgebras of a monad $M$ on $\Set$. Moreover, the category of $M$-bialgebras admits a substitution (non-symmetric) monoidal structure. This allows us to define `$M$-plethories' as monoids with respect to this monoidal structure \cite{StaceyWhitehouse}. Then we use this perspective to categorify the notion of rig-plethory, obtaining the concept of 2-plethory. In \cref{thm:ksbar_as_2-plethory} we give $\Schur$ the structure of a 2-plethory.

In \cref{sec:lambdaplus} we begin the decategorification process by studying the rig of isomorphism classes of objects in $\Schur$. We denote this rig by $\Lambda_+$, and call its elements `positive symmetric functions', since it is a sub-rig of the famous ring of symmetric functions, $\Lambda$. In \cref{thm:lambdaplus-forms-birig} we equip $\Lambda_+$ with a birig structure using the 2-birig structure on $\Schur$, and in \cref{thm:lambdaplus-forms-rig-plethory} we equip $\Lambda_+$ with a rig-plethory structure using the 2-plethory structure on $\Schur$.

In \cref{sec:lambda} we study the group completion of $\Lambda_+$, which is $\Lambda$. This is evidently a ring, but making it into a biring is less straightforward: to define co-negation in this biring and prove its properties we need the homology of $\Z_2$-graded chain complexes of Schur functors. We make $\Lambda$ into a biring in \cref{thm:lambda_biring}, and make it into a ring-plethory in \cref{thm:lambda_plethory}.

\subsection*{Notation}

We use a sans serif font for 1-categories, e.g.\ $\Cat$, and a bold serif font for 2-categories, e.g.\ $\CCat$. On occasion we need to think of a 1-category as a locally discrete 2-category, in which case we do not change the font, and hope it is clear by context. Other times, we need to modify a 2-category to give a 1-category. In these cases, we do not change the font, but merely decorate the name of the 2-category to indicate what was done, e.g.\ $\CCat\ho$.

\subsection*{Acknowledgements}

We dedicate this paper to Andr\'e Joyal, whose work has been so inspiring to us.  We would like to thank Abdelmalek Abdesselam, Martin Brandenburg, Richard Garner, Julia Ramos Gonz\'alez, Steve Lack, Jade Master, Jackie Shadlen, Mike Shulman, David Spivak, Benjamin Steinberg, and Christian Williams for helpful conversations, and the referees for extensive constructive feedback.  Special thanks go to Allen Knutson as the `Primary Instigator'. The work in this paper was completed while the second author was affiliated with UCR, not with NIST.

\section{Schur functors}
\label{sec:Schur}

We begin this section with a brief overview of the classical theory of Schur functors and its relation to the representation theory of the symmetric groups. In fact, we do not need and do not use any of this classical material to develop our account. Thus, \cref{sec:classical} is included merely to smooth the transition from older more established accounts, for those readers already familiar with them, to the new approach based more on categorical principles which we set out starting in \cref{sec:abstract}.

\subsection{Classical treatment}
\label{sec:classical}

Classically, a Schur functor is a specific sort of functor
\[  
    F \maps \Fin\Vect \to \Fin\Vect
\]
where $\Fin\Vect$ is the category of finite-dimensional vector spaces over some fixed field $k$ of characteristic zero. 
Namely, it is a functor obtained from an irreducible representation of the symmetric group $S_n$ by sending a vector space $V$ to a certain subspace of $V^{\otimes n}$ as we describe below.
%

Irreducible representations of $S_n$ correspond to $n$-box Young diagrams, so Schur functors are usually described with the help of these. An $n$-box Young diagram is simply a way to write $n$ as a sum of natural numbers listed in decreasing order. For example, this 17-box Young diagram:
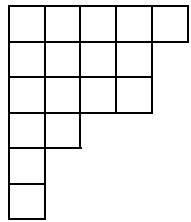
\[
    \yng(5,4,4,2,1,1) 
\]
describes the partition of $17$ as $5 + 4 + 4 + 2 + 1 + 1$. It also can be used to construct an irreducible complex representation of the symmetric group $S_{17}$, and thus a Schur functor. 

Classically, the relation between Young diagrams and Schur functors is often described using the group algebra of the symmetric group, $k[S_n]$. Given an $n$-box Young diagram $\lambda$, we can think of the operation `symmetrize with respect to permutations of the boxes in each row' as an element $p^S_\lambda \in k[S_n]$. Similarly, we can think of the operation `antisymmetrize with respect to permutations of the boxes in each column' as an element $p^A_\lambda \in k[S_n]$. By construction, each of these elements is idempotent. They do not commute, but their product $p^A_\lambda p^S_\lambda$, times a suitable nonzero constant, gives an idempotent $p_\lambda$.

The element $p_\lambda \in k[S_n]$ is called the \define{Young symmetrizer} corresponding to the $n$-box Young diagram $\lambda$. Since the group algebra $k[S_n]$ acts on $V^{\otimes n}$ by permuting the factors, the Young symmetrizer gives a projection
\[ 
    p_\lambda \maps V^{\otimes n} \to V^{\otimes n} 
\]
whose image is a subspace called $S_\lambda(V)$. Since $p_\lambda$ commutes with everything in $k[S_n]$, this subspace is invariant under the action of $S_n$, and it is a direct sum of copies of a specific irreducible representation of $S_n$. But the point is this: there is a functor
\[  
    S_\lambda \maps \Fin\Vect \to \Fin\Vect 
\]
which sends a space $V$ to the space $S_\lambda(V) \subseteq V^{\otimes n}$. This is called a `Schur functor'.  Departing slightly from tradition, we could call any finite direct sum of functors of the form $S_\lambda$ a `Schur functor'. We have many examples:

\begin{itemize}
    \item For each $n \geq 0$, the $n$th tensor power $V \mapsto V^{\otimes n}$ is a Schur functor. 
    \item If $F$ and $G$ are Schur functors, the functor $V \mapsto F(V) \oplus G(V)$ is a Schur functor.
    \item If $F$ and $G$ are Schur functors, the functor $V \mapsto F(V) \otimes G(V)$ is also a Schur functor.
    \item If $F$ and $G$ are Schur functors, the composite $V \mapsto F(G(V))$ is a Schur functor. This way of constructing Schur functors is known as \define{plethysm}.
    \item For any right $k[S_n]$-module $\rho$, there is a Schur functor $V \mapsto \rho \otimes_{k[S_n]} V^{\otimes n}.$
\end{itemize}

The last example illustrates an idea that becomes crucial in the next section.

\subsection{Abstract Schur functors}
\label{sec:abstract}

So far we have given a classical account of Schur functors as special endofunctors on $\Fin\Vect$.  However, Schur functors make sense much more broadly. We now show that they give endofunctors on any 2-rig. A somewhat novel feature of our treatment is that we \emph{do not require the theory of Young diagrams} to define Schur functors: instead we use 2-rigs and some notions from the theory of 2-categories.  From the realm of representation theory we need only Maschke's theorem---until \cref{lem:J-preserves-coproducts}, where we also use the fact that any field of characteristic zero is a splitting field for the symmetric groups.

Our strategy is as follows. In
\cref{sec:schur_functors_from_polynomial_species} we introduce `polynomial species' and show how any polynomial species $\rho$ gives an endofunctor $F_{\rho,R} \maps \R \to \R$ on any 2-rig $\R$.  Moreover, this endofunctor depends pseudonaturally on $\R$.  Following this cue, in \cref{sec:abstract_schur_functors} we define a Schur functor to be a pseudonatural transformation from the forgetful 2-functor $U \maps \TRig \to \Cat$ to itself.  Then, in \cref{sec:Poly}, we show that \emph{every} such Schur functor arises from a polynomial species.

\subsubsection{Schur functors from polynomial species}
\label{sec:schur_functors_from_polynomial_species}

In his famous paper introducing combinatorial species,  Joyal \cite{Especies} also introduced `linear' species, which are functors from the groupoid of finite sets to $\Vect$. We begin by describing a special kind of linear species which we call `polynomial species'.  Then we show how any polynomial species gives an endofunctor on any 2-rig.  

To define polynomial species it is convenient to work with a skeleton of the groupoid of finite sets, namely the \define{symmetric groupoid} $\S$ where objects are natural numbers, all morphisms are automorphisms, and the automorphisms of the object $n$ form the group $S_n$. 

\begin{defn} 
    A \define{polynomial species} is a functor $F \maps \S\op \to \Fin\Vect$ such that $F(n) \cong 0$ for all sufficiently large $n$. Let $\Poly$ be the category where objects are polynomial species and morphisms are natural transformations.
\end{defn}

Thanks to the `op', polynomial species can be seen as special $\Vect$-valued presheaves on $\S$.  However, since $\S$ is a groupoid we have $\S \cong \S\op$.  This allows us to treat $\Poly$ as a subcategory of the category of \define{representations} of $\S$, by which we simply mean functors $F \maps \S \to \Vect$. Every irreducible representation of $\S$ is finite-dimensional: it is really just an irreducible representation of some group $S_n$. Every representation of $\S$ is a direct sum of irreducibles. The category $\Poly$ may thus be identified with the full subcategory consisting of finite direct sums of irreducibles.  However, including the `op' means that for each $n$, $\rho(n)$ most naturally becomes a \emph{right} module of the group algebra $k[S_n]$.

We claim that for any polynomial species $\rho \maps \S \to \Fin\Vect$ and any 2-rig $\R$ there is a functor $F_{\rho,\R} \maps \R \to \R$ that sends any object $x \in \R$ to
    \[   
        F_{\rho, \R}(x) = \bigoplus_n \rho(n) \otimes_{k[S_n]} x^{\otimes n} .
    \]
However, we need to explicate the meaning of this expression.  The group algebra $k[S_n]$ is a monoid in $\Vect$, and $\rho(n)$ is a right $k[S_n]$-module.  How then are we forming the tensor product $\rho(n) \otimes_{k[S_n]} x^{\otimes n}$ as an object in $\R$?  First, note:

\begin{lem}
\label{prop:base_change}
    For any 2-rig $\R$ there is exactly one symmetric monoidal linear functor $i \maps \Fin\Vect \to \R$, up to monoidal linear natural isomorphism.
\end{lem}

\begin{proof}
    Let $\Mat$ be the linear category whose objects are integers $m \geq 0$ and whose morphisms $m \to n$ are $m \times n$ matrices with entries in $k$. Since $\R$ is Cauchy complete and in particular has finite biproducts, there is an evident linear functor 
    \[
        \Mat \to \R
    \] 
    which takes $m$ to $I^m$, the direct sum of $m$ copies of the tensor unit $I$. It is the unique linear functor taking $1$ to $I$, up to unique linear isomorphism. In the case $\R = \Fin\Vect$, the linear functor 
    \[
        \Mat \to \Fin\Vect
    \] 
    taking $1$ to $k$ is a linear equivalence (exhibiting $\Mat$ as a skeleton of $\Fin\Vect$). Because of this equivalence, we could equally well say that there is a linear functor 
    \[
        i \maps \Fin\Vect \to \R
    \] 
    which, up to unique linear isomorphism, is the unique linear functor taking $k$ to $I$. Notice that a symmetric monoidal functor of this form must take the tensor unit $k$ to $I$ (up to coherent isomorphism, as always), and in fact $i$ is symmetric monoidal, because there is a canonical isomorphism
    \[
        I^m \otimes I^n \cong I^{m n}, 
    \] 
    using the fact that $\otimes$ preserves direct sums in each argument, and the fact that there is a canonical isomorphism $I \otimes I \cong I$. 
\end{proof}

Using \cref{prop:base_change}, for any 2-rig $\R$ we can apply the functor $i \maps \Fin\Vect \to \R$ to the group algebra $k[S_n]$ and obtain a monoid in $\R$, which by abuse of notation we again call $k[S_n]$.   This monoid is isomorphic to a direct sum of copies of the unit object $I \in \R$, one for each permutation $\sigma \in \S_n$:
\[    k[S_n] \cong \bigoplus_{\sigma \in S_n} I \in \R. \]

For any object $x \in \R$, $x^{\otimes n}$ is a left module of $k[S_n]$, in the same manner as it would be in $\Fin\Vect$.  Using \cref{prop:base_change}, any right $k[S_n]$-module $A$ in $\Fin\Vect$ gives a right $k[S_n]$-module in $\R$, which we again call $A$.  Putting these together, we can use a coequalizer to define an object of $\R$ that we call $A \otimes_{k[S_n]} x^{\otimes n}$.  The key point is that while 2-rigs may not have all colimits, we can build $A \otimes_{k[S_n]} x^{\otimes n}$ using absolute colimits (which in this context means colimits that are preserved by any linear functor), as explained below.

\begin{lem}
\label{lem:group_algebra}
Suppose $A$ is a finitely generated right $k[S_n]$-module.  If we treat $A$ as module of $k[S_n]$ in $\mathsf{R}$ via $\Fin\Vect \to \mathsf{R}$, the following diagram has a coequalizer:
\[
    \begin{tikzcd}
        A \otimes k[S_n] \otimes x^{\otimes n} 
        \arrow[r, shift left = 0.3em, ""]
        \arrow[r, swap, shift right = 0.3em, ""]
        &
        A \otimes x^{\otimes n}
    \end{tikzcd}
    \]
where the top arrow comes from the right action of $k[S_n]$ on $A$, and the bottom arrow comes from the left action of $k[S_n]$ on $x^{\otimes n}$.
\end{lem}

\begin{proof}
We proceed in stages. The first stage is where $A$ is the right regular representation $k[S_n]$; the second is where $A$ is a finite sum of copies $N \cdot k[S_n]$ of the regular representation $k[S_n]$, and the third stage is where $A$ is a quotient of some $N \cdot k[S_n]$. In each of these cases, $r \maps A \otimes k[S_n] \to A$ denotes the right action, and $l \maps k[S_n] \otimes x^{\otimes n} \to x^{\otimes n}$ denotes the left action. 

For the first stage, where $A = k[S_n]$, the diagram 
\[
\begin{tikzcd}
    k[S_n] \otimes k[S_n] \otimes x^{\otimes n} 
        \arrow[r, shift left = 0.3em, "r \otimes 1"]
        \arrow[r, swap, shift right = 0.3em, "1 \otimes l"]
        &
        k[S_n] \otimes x^{\otimes n} \arrow[r, "l"] & x^{\otimes n}
\end{tikzcd}
\] 
is a coequalizer, in fact a split coequalizer, where the splitting is given by maps $x^{\otimes n} \to k[S_n] \otimes x^{\otimes n}$ and $k[S_n] \otimes x^{\otimes n} \to k[S_n] \otimes k[S_n] \otimes x^{\otimes n}$. The first such map is 
\[
\begin{tikzcd}
    x^{\otimes n} \cong I \otimes x^{\otimes n} \arrow[r, "u \otimes 1"] & k[S_n] \otimes x^{\otimes n}
\end{tikzcd} 
\] 
where the isomorphism is the left unitor and $u \maps I \to k[S_n]$ names the unit element of $k[S_n]$. The second map of the splitting is defined similarly, using a map of type $u \otimes 1 \otimes 1$.  Since a split coequalizer is an absolute colimit, this coequalizer exists in $\R$.

For the second stage, let $N \in \mathbb{N}$ and apply the functor $R \mapsto N \cdot R$ to the split coequalizer from the first stage.  The result is a split coequalizer naturally isomorphic to one of this form:
\[
\begin{tikzcd}
    (N \cdot k[S_n]) \otimes k[S_n] \otimes x^{\otimes n} 
        \arrow[r, shift left = 0.3em, "r \otimes 1"]
        \arrow[r, swap, shift right = 0.3em, "1 \otimes l"]
        &
        (N \cdot k[S_n]) \otimes x^{\otimes n} \arrow[r] & N \cdot x^{\otimes n}.
\end{tikzcd}
\]
This completes the second stage. 

For the third stage, by Maschke's theorem every $k[S_n]$-module $A$ is the result of splitting an idempotent $e \maps N \cdot k[S_n] \to N \cdot k[S_n]$. From the second stage, there is a morphism of coequalizer diagrams 
\[
\begin{tikzcd}
    (N \cdot k[S_n]) \otimes k[S_n] \otimes x^{\otimes n} 
        \arrow[r, shift left = 0.3em, "r \otimes 1"] \arrow[d, "e \otimes 1 \otimes 1"]
        \arrow[r, swap, shift right = 0.3em, "1 \otimes l"]
        &
        (N \cdot k[S_n]) \otimes x^{\otimes n} \arrow[r] \arrow[d, "e \otimes 1"] & N \cdot x^{\otimes n} \arrow[d, "\tilde{e}"] \\
    (N \cdot k[S_n]) \otimes k[S_n] \otimes x^{\otimes n} 
        \arrow[r, shift left = 0.3em, "r \otimes 1"] 
        \arrow[r, swap, shift right = 0.3em, "1 \otimes l"]
        &
        (N \cdot k[S_n]) \otimes x^{\otimes n} \arrow[r] 
        & N \cdot x^{\otimes n}    
\end{tikzcd}
\] 
for some idempotent $\tilde{e}$. View this morphism of diagrams as a coequalizer diagram in the category of functors $\mathsf{R}^{\I}$ from the `generic idempotent' $\I$ (one object, two morphisms, both idempotent). Since splitting of idempotents, as a functor $\mathrm{colim} \maps \mathsf{R}^I \to \mathsf{R}$, is left adjoint to the evident diagonal functor $\mathsf{R} \to \mathsf{R}^{\I}$, it maps the above coequalizer diagram in $\mathsf{R}^{\I}$ to a coequalizer diagram in $\mathsf{R}$, namely
\[
\begin{tikzcd}
    A \otimes k[S_n] \otimes x^{\otimes n} 
       \arrow[r, shift left = 0.3em, "r \otimes 1"] \arrow[r, swap, shift right = 0.3em, "1 \otimes l"] & 
    A \otimes x^{\otimes n} \arrow[r] & 
    A \otimes_{k[S_n]} x^{\otimes n}.  
\end{tikzcd} \qedhere
\]
\end{proof} 

\begin{prop} 
For any polynomial species and any $\rho \maps \S \to \Fin\Vect$ and any 2-rig $\R$ there is a functor $F_{\rho,\R} \maps \R \to \R$ given as follows:
    \[   
        F_{\rho, \R}(x) = \bigoplus_n \rho(n) \otimes_{k[S_n]} x^{\otimes n}  
    \]
for any object $x \in \R$, and likewise
on morphisms.
\end{prop}

Next suppose that we have a symmetric monoidal linear functor $G \maps \R \to \R'$ between 2-rigs. We can think of $G$ as a `change of base category'. We now show that the functors $F_{\rho,R}$ are natural, or more precisely pseudonatural, with respect to change of base. 

\begin{prop} 
\label{prop:Schurnatural}
    For a polynomial species $\rho$ and a symmetric monoidal linear functor $G \colon \R \to \R'$, there is a natural isomorphism $\phi_G \maps G \circ F_{\rho, \R} \xrightarrow{\sim} F_{\rho, \R'} \circ G$.
\end{prop}

\begin{proof}
    By definition of symmetric monoidal functor, $G$ preserves tensor products up to specified coherent isomorphisms, and $G$ will automatically preserve both direct sums (by linearity), and split coequalizers (as all functors do). In other words, for any polynomial functor $\rho$ we have natural isomorphisms
    \[   
    \begin{array}{ccl}
        F_{\rho, \R}(Gx) &=& \displaystyle{ \bigoplus_n \rho(n) \otimes_{k[S_n]} (Gx)^{\otimes n} } \\ \\
        &\cong& \displaystyle{ \bigoplus_n \rho(n) \otimes_{k[S_n]} G(x^{\otimes n}) } \\ \\
        &\cong& \displaystyle{ \bigoplus_n G(\rho(n) \otimes_{k[S_n]} x^{\otimes n}) } \\ \\
        &\cong& \displaystyle{ G \big( \bigoplus_n
        \rho(n) \otimes_{k[S_n]} x^{\otimes n}\big) }
    \end{array}
    \]
where we use the fact from the proof of \cref{lem:group_algebra} that all colimit constructions in view are absolute colimits.
\end{proof}

\subsubsection{Definition of abstract Schur functors}
\label{sec:abstract_schur_functors}

We have seen how to define an endofunctor 
$F_{\rho, \R} \maps \R \to \R$  for any polynomial species $\rho$ and any 2-rig $\R$, and we have seen that if $G \maps \R \to \R'$ is a symmetric monoidal linear functor between 2-rigs, the diagram 
\[
\begin{tikzcd}
    \R
    \arrow[r, "G"]
    \arrow[d, swap, "F_{\rho, \R}"]
    &
    \R'
    \arrow[d, "F_{\rho, \R'}"]
    \\
    \R
    \arrow[r, swap, "G"]
    &
    \R'
\end{tikzcd}
\] 
commutes up to a natural isomorphism $\phi_G \maps F_{\rho, \R} \circ G \xrightarrow{\sim} G \circ F_{\rho, \R'}$.   

It may be wondered what role is played by these natural isomorphisms $\phi_G$.  As we shall see, they express the fact that the endofunctors $F_{\rho, \R}$ depend pseudonaturally on $\R$.  The skeptic might reply: this is pleasant to observe---but surely it just some piddling abstract nonsense in the larger story of Schur functors, which are, after all, deeply studied and incredibly rich classical constructions?   

Let us put the question another way: among endofunctors on 2-rigs that depend pseudonaturally on the 2-rig, what is special about the endofunctors $F_{\rho,\R}$ arising from representations of the symmetric group? What extra properties pick out exactly \emph{these} functors?  

The perhaps surprising answer is: no extra properties! The functors $F_{\rho,\R}$ are \emph{precisely} those functors that are defined on all 2-rigs and that are pseudonatural with respect to maps between 2-rigs.  

We now make this precise. These functors $F_{\rho,\R}$ are defined on certain symmetric monoidal linear categories, but they respect neither the symmetric monoidal structure nor the linear structure. So, we have to forget some of the structure of the objects on which these functors are defined. This focuses our attention on the `forgetful' 2-functor
\[
    U \maps \TRig \to \CCat
\] 
where:
\begin{defn}
\label{defn:2-rig}
    Let $\TRig$ denote the 2-category with
    \begin{itemize}
        \item symmetric monoidal Cauchy complete linear categories as objects, 
        \item symmetric monoidal linear functors as morphisms, 
        \item symmetric monoidal linear natural transformations as 2-morphisms.
    \end{itemize}
\end{defn}
We now propose our conceptual definition of the category of Schur functors:

\begin{defn}
    A \define{Schur functor} is a pseudonatural transformation $S \maps U \To U$, where $U \maps \TRig \to \CCat$ is the forgetful 2-functor. A \define{morphism} of Schur functors is a modification between such pseudonatural transformations. Let $\Schur = [U, U]$ be the category with Schur functors as objects and modifications between these as morphisms.
\end{defn}

What this proposed definition makes manifestly obvious is that \emph{Schur functors are closed under composition}. This will provide a satisfying conceptual explanation of `plethysm'.

\section{Equivalence with polynomial species}
\label{sec:Poly}

Our first main result, \cref{thm:schur_versus_poly}, will be that $\Schur$ is equivalent to $\Poly$. But before launching into the proof, it is worth pondering an easier problem where we replace categories by sets and 2-rigs by rings. So, let $\Ring$ be the category of rings (commutative, as always). There is a forgetful functor 
\[ 
    U \maps \Ring \to \Set.
\]
What are the natural transformations from this functor to itself? Any polynomial $P \in \Z[x]$ defines such a natural transformation, since for any ring $R$ there is a function $P_R \maps U(R) \to U(R)$ given by
\[ 
    P_R \maps x \mapsto P(x) 
\]
and this is clearly natural in $R$. But in fact, the set of natural transformations from this functor turns out to be \emph{precisely} $\Z[x]$. And the reason is that $\Z[x]$ is the free ring on one generator!  

To see this, note that the forgetful functor
\[ 
    U \maps \Ring \to \Set  
\]
has a left adjoint, the `free ring' functor
\[ 
    F \maps \Set \to \Ring  . 
\] 
The free ring on a 1-element set is
\[ 
    F(1) \cong \Z[x] 
\]
and homomorphisms from $F(1)$ to any commutative ring $R$ are in one-to-one correspondence with elements of the underlying set of $R$, since
\[ 
    U(R) \cong \Set(1, U(R)) \cong \Ring(F(1), R). 
\]
So, we say $F(1)$ represents the functor $U$. This makes it easy to show that the set of natural transformations from $U$ to itself, denoted $[U, U]$, is isomorphic to the underlying set of $\Z[x]$, namely $U(F(1))$:
\[
     U(F(1)) \cong \Ring(F(1), F(1)) \cong [\Ring(F(1), -), \Ring(F(1), -)] \cong [U, U]. 
\] 
In the first step here we use the representability $U \cong \Ring(F(1), -)$, in the second we use the Yoneda lemma, and in the third we use the representability again.

We shall carry out a categorified version of this argument to prove that $\Schur$ is equivalent to $\Poly$. The key will be showing that just as $\Z[x]$ is the free ring on one generator, $\Poly$ is the free 2-rig on one generator. 

\begin{thm}
\label{thm:schur_versus_poly}
    There is an equivalence of categories $\Poly \to \Schur$ which sends any polynomial species $\rho \maps \S \to \Fin\Vect$ to the Schur functor $F_\rho$ defined by the formula 
    \[   
        F_{\rho, \R}(x) = \bigoplus_n \rho(n) \otimes_{k[S_n]} x^{\otimes n}  
    \]
    for any 2-rig $\R$ and any object $x \in \R$. 
\end{thm}

In what follows we use $k \S$ to denote the `$k$-linearization' of the symmetric groupoid: that is, the linear category formed by replacing the homsets in $\S$ by the free $k$-vector spaces on those homsets. This means that for each object $n$ of $k \S$, we may speak of the representable functor $k \S(-, n) \maps k \S^{\op} \to \Vect$. We use $\overline{k \S}$ to denote the Cauchy completion of $k \S$ as a linear category. As we shall see, $\ksbar$ is the free 2-rig on one generator. To be careful, we use $U(\ksbar)$ to denote its underlying category. We construct the equivalence in \cref{thm:schur_versus_poly} in several steps:
\[ 
\begin{array}{cll}
    \Poly 
    &\simeq U(\ksbar) 
    &\text{\cref{thm:poly_versus_kS}}
    \\&\simeq \TRig(\ksbar, \ksbar) 
    &\text{\cref{thm:representing_U}}
    \\&\simeq [\TRig(\ksbar, -), \TRig(\ksbar, -)] 
    &\text{2-categorical Yoneda Lemma}
    \\&\simeq [U, U]= \Schur
    &\text{\cref{thm:representing_U}.}
\end{array}
\]
The equivalences here are equivalences of categories, but since $\ksbar$ is a 2-rig they give a way to make $\Poly$ and $\Schur$ into 2-rigs as well. We describe the resulting 2-rig structure on $\Poly$ in \cref{prop:poly_2rig}, and on $\Schur$ in \cref{prop:schur_2rig}.

In general, the Cauchy completion (or `Karoubi envelope') $\overline{\C}$ of a linear category $\C$ consists of the full subcategory of linear functors $\C\op \to \Vect$ that are retracts of finite direct sums of representables $\C(-, c) \maps \C\op \to \Vect$ \cite[Cor.\ 4.22]{LackTendas}. In the case $\C = k\S$ we can identify this Cauchy completion with $\Poly$ as follows:

\begin{thm}
\label{thm:poly_versus_kS}
    The functor $\Poly \to U(\ksbar)$ which sends a polynomial species $\rho \maps \S \to \Fin\Vect$ to its unique extension to a linear functor $k\S \to \Fin\Vect$ is an equivalence.
\end{thm}

\begin{proof}
    Since $\S$ is a groupoid we can identify a polynomial species $\rho \maps \S \to \Fin\Vect$ with a functor $\S\op \to \Fin\Vect$, and this in turn extends uniquely to a linear functor $k\S\op \to \Fin\Vect$. Every polynomial species $\rho$ is a finite coproduct $\bigoplus_{j = 0}^n \rho(j)$ where $\rho(j) \maps \S \to \Fin\Vect$ vanishes on all objects $i \ne j$. By Maschke's theorem, each representation of $S_j$ is the retract of a finite sum of copies of the group algebra $k[S_j]$, which corresponds to the representable $k\S(-, j)$. Thus, the polynomial species correspond precisely to the linear functors $k\S\op \to \Fin\Vect$ that are in the Cauchy completion $\ksbar$.
\end{proof}

Next we prove that $\ksbar$ is the free 2-rig on one generator.  For previous results of a similar flavor, see \cite[Ex.\ 1.26]{MD82} and \cite[Prop. 3.3]{DavydovMolev}.  We use this fact to show that $\ksbar$ represents the forgetful 2-functor from 2-rigs to categories

\begin{thm}
\label{thm:representing_U}
    The forgetful 2-functor 
    \[
        U \maps \TRig \to \CCat
    \] 
    has a left 2-adjoint $F \maps \CCat \to \TRig$, and $U$ is represented by $F(1) = \overline{k \S}$. In other words:
    \[   
        \TRig(\overline{k \S}, -) \simeq U(-)
    \] 
    This equivalence sends any morphism of 2-rigs $\phi \maps \ksbar \to \R$ to the object $\phi(1)$ in $U(\R)$.
\end{thm}

As a first step toward this, we write $U$ as a composite of three functors, each of which have left 2-adjoints:
\[
\begin{tikzcd}
    \CCat \phantom{\;}
    \arrow[r, bend left, "\SS"]
    \arrow[r, phantom, "\bot", pos = 0.6]
    &
    \SMC
    \arrow[l, bend left, "U_0"]
    \arrow[r, bend left, "k(-)"]
    \arrow[r, phantom, "\bot"]
    &
    \SMLin
    \arrow[l, bend left, "U_1"]
    \arrow[r, bend left, "\overline{(-)}"]
    \arrow[r, phantom, "\bot", pos = 0.4]
    &
    \TRig
    \arrow[l, bend left, "U_2", pos = 0.45]
\end{tikzcd}
\]
The 2-rig $\ksbar$ is obtained by applying the composite of these left adjoints to the terminal category $1$, so we say $\ksbar$ is the free 2-rig on one generator.  We construct these 2-adjunctions in the lemmas below. We shall need a number of 2-categories:

\begin{defn}
\label{defn:CauchLin}
Let $\SMC$ be the 2-category of 
\begin{itemize}
    \item symmetric monoidal categories, 
    \item (strong) symmetric monoidal functors, and 
    \item monoidal natural transformations.
\end{itemize}
Let $\Lin$ be the 2-category of
\begin{itemize}
    \item linear categories, 
    \item linear functors, and
    \item linear natural transformations.
\end{itemize}
Let $\SMLin$ be the 2-category of 
\begin{itemize}
    \item symmetric monoidal linear categories, 
     \item symmetric monoidal linear functors, and
    \item symmetric monoidal linear natural transformations.
\end{itemize}
Let $\Cauch\Lin$ be the 2-category of 
\begin{itemize}
    \item Cauchy complete linear categories, 
    \item linear functors, and 
    \item linear natural transformations. 
\end{itemize}
\end{defn}

\begin{lem}
\label{lem:freeSMC}
    The forgetful 2-functor
    \[ 
        U_0 \maps \SMC \to \CCat 
    \]
    has a left 2-adjoint
    \[  
        \SS \maps \CCat \to \SMC 
    \]
    such that $\SS(1)$ is equivalent as a symmetric monoidal category to $\S$. In particular, the functor $x \maps 1 \to U_0\C$ which picks out an object $x$ in a symmetric monoidal category $\C$ corresponds to the symmetric monoidal functor $x' \maps \S \to \C$ with $n \mapsto x^{\otimes n}$.
\end{lem}
\begin{proof}
    See \cite{2-dimmonadtheory}, and also \cite[Sec.\ 4.1]{GeneralisedSpecies}. 
\end{proof}

\begin{lem}
\label{lem:U1}
    The forgetful 2-functor
    \[ 
        U_1 \maps \SMLin \to \SMC 
    \]
    has a left 2-adjoint
    \[ 
        k(-) \maps \SMC \to \SMLin. 
    \]
    In particular, the symmetric monoidal functor $x' \maps \S \to U_1\C$ such that $x'(n)$ is the object $x^{\otimes n}$ in a symmetric monoidal linear category $\C$ corresponds to the symmetric monoidal linear functor $x'' \maps k\S \to \C$ such that $n \mapsto x^{\otimes n}$.
\end{lem}
\begin{proof}
    The underlying 2-functor $\Lin \to \CCat$ has a left 2-adjoint $k(-) \maps \CCat \to \Lin$  \cite[Prop.\ 6.4.7]{BorceuxII}. Given any category $\C$, $k\C$ is the linear category with the same objects whose hom-spaces are the free vector spaces on the homsets of $\C$. The 2-functor $k(-)$ is also given by `change of base' along the strong symmetric monoidal functor sending any set to the free vector space on that set. It therefore sends symmetric pseudomonoids \cite{Monoidalbicatshopfalgebroids} in $\CCat$ to symmetric pseudomonoids in $\Lin$. In other words, it sends symmetric monoidal categories to symmetric monoidal linear categories. Therefore, the 2-adjunction $k(-) \dashv U_0$ between $\CCat$ and $\Lin$ lifts to one between $\SMC$ and $\SMLin$:  
    \[
    \begin{tikzcd}
        \SMC
        \arrow[r, bend left, "k(-)"]
        \arrow[r, phantom, "\bot"]
        &
        \SMLin
        \arrow[l, bend left, "U_1"]
    \end{tikzcd}
    \qedhere
    \]
\end{proof}

\begin{lem}
\label{lem:U2}
    The forgetful 2-functor
    \[ U_2 \maps \TRig \to \SMLin 
    \]
    has a left 2-adjoint
    \[  \overline{(-)} \maps \SMLin \to \TRig \]
    making $\TRig$ into a reflective sub-2-category of
    $\SMLin$. In particular, for a 2-rig $\C$, a symmetric monoidal linear functor $x'' \maps k\S \to U_2\C$ that picks out an object $x$ in the underlying symmetric monoidal linear category of $\C$ corresponds to a unique symmetric monoidal linear functor $\overline x \maps \ksbar \to \C$ such that $\rho \mapsto \bigoplus_n \rho(n) \otimes_{k[S_n]} x^{\otimes n}$.
\end{lem}
\begin{proof}
      Cauchy completion gives a 2-reflector $\overline{(-)} \maps \Lin \to \Cauch\Lin$ which is left 2-adjoint to the 2-embedding $i \maps \Cauch\Lin \to \Lin$, and we claim the 2-adjunction $\overline{(-)} \dashv i$ lifts to the level of symmetric monoidal structure to give a 2-adjunction 
    \[
    \begin{tikzcd}
        \SMLin
        \arrow[r, bend left, "\overline{(-)}"]
        \arrow[r, phantom, "\bot", pos = 0.45]
        &
        \TRig.
        \arrow[l, bend left, "U_2", pos = 0.45]
    \end{tikzcd}
    \]
    
    To prove this, we introduce a tensor product $\boxtimes$ appropriate to Cauchy complete linear categories $\C, \D$: it is simply the Cauchy completion of the tensor product of their underlying linear categories:
    \[   \C \boxtimes \D = \overline{U\C \otimes U\D}. \]
    Recall that the Cauchy completion of a linear category is the full subcategory of $\Vect$-valued presheaves on that category that are retracts of finite coproducts of representables \cite[Cor.\ 4.22]{LackTendas}.    Objects of $\C \boxtimes \D$ are thus retracts of finite coproducts of tensor products $\bigoplus_i c_i \otimes d_i$ where such sums are formally defined as coproducts of linear functors $\C(-, c_i) \otimes \D(-, d_i) \maps (U\C \otimes U\D)\op \to \Vect$. It is worth remarking that `bilinearity relations', e.g.
    \[ 
        (c \oplus c') \otimes d \cong (c \otimes d) \oplus (c' \otimes d), 
    \]
    do not need to be imposed, but are automatically built into the definition of $\C \boxtimes \D$, due to the absoluteness of coproducts as $\Vect$-enriched colimits.
    
    A useful fact is that for linear categories $\A, \B$, there is a canonical enriched functor  $\overline{\A} \boxtimes \overline{\B} \simeq \overline{\A \otimes \B}$. This amounts to asserting a canonical equivalence 
    \[ 
        \overline{U\overline{\A} \otimes U\overline{\B}} \simeq \overline{\A \otimes \B}.
    \]
    To prove this, it suffices to show that $U\overline{\A} \otimes U\overline{\B}$ and $\A \otimes \B$ have equivalent enriched presheaf categories \cite[Proposition 5.28]{Kelly}. But 
    \begin{align*}
        [(U\overline{\A} \otimes U\overline{\B})\op, \Vect]
        &\simeq [U\overline{\A}\op \otimes U\overline{\B}\op, \Vect]
        \\&\simeq [U\overline{\A}\op, [U\overline{\B}\op, \Vect]]
        \\&\simeq [U\overline{\A}\op, [\B\op, \Vect]]
        \\&\simeq [\A\op, [\B\op, \Vect]]
        \\&\simeq [(\A \otimes \B)\op, \Vect]
    \end{align*}
    because the enriched presheaf category of any enriched category is equivalent to that of its Cauchy completion \cite{Kelly, AbsoluteColimEnriched}. Thus, Cauchy completion becomes a strong monoidal 2-functor $\overline{(-)} \maps (\Lin, \otimes) \to (\Cauch\Lin, \boxtimes)$. Even better, this 2-functor is symmetric monoidal, so it sends symmetric pseudomonoids to symmetric pseudomonoids. In other words, it sends symmetric monoidal linear categories to symmetric monoidal Cauchy complete linear categories.
        
    To check that the formula given for the extension $\overline x \maps \ksbar \to \C$ of $x'' \maps k\S \to \C$ to the Cauchy completion is correct, we check that it matches $x''$ when we apply it to representables. On a representable, we have
    \begin{align*}
        k\S(m, -) 
        & \mapsto \bigoplus_n k\S(m, n) \otimes_{k[S_n]} x^{\otimes n}
        \\&\cong k[S_m] \otimes_{k[S_m]} x^{\otimes m}
        \\&\cong x^{\otimes m} 
        \\&= x''(m). \qedhere  
    \end{align*}
\end{proof}

\begin{proof}[Proof of \cref{thm:representing_U}]
    The underlying 2-functor $U \maps \TRig \to \CCat$ is the composite 
    \[
        \TRig \xrightarrow{\; U_2\; } \SMLin \xrightarrow{\; U_1 \;} \SMC \xrightarrow{\; U_0 \;} \CCat
    \] 
    and thus by the above lemmas, we have pseudonatural equivalences 
    \begin{align*}
        \TRig( \ksbar, -) 
        & \cong \SMLin(k\S, U_2 -) \\
        & \cong \SMC(\S, U_1 U_2 -) \\
        & \cong U_0 U_1 U_2 \\
        & \cong U
    \end{align*}
    so that $\ksbar$ is the representing object for $U$. 
\end{proof}

\begin{proof}[Proof of \cref{thm:schur_versus_poly}]
    First we describe a chain of equivalences
     \[ 
        \Poly 
        \simeq U(\ksbar) 
        \simeq \TRig(\ksbar, \ksbar) 
        \simeq [\TRig(\ksbar -), \TRig(\ksbar, -)]  
        \simeq [U, U]= \Schur   
    \]
    and then we explicitly describe the Schur functor corresponding to a polynomial species. Let $\rho \in \Poly$ and let $\R$ be any 2-rig. 
    
    \begin{itemize}
        \item We have $\Poly \simeq U(\ksbar)$ by \cref{thm:poly_versus_kS}. Denote the counterpart of $\rho$ under this equivalence by $\widetilde \rho \in U(\ksbar)$. 
        \item     
        We have $U(\ksbar) \simeq \TRig(\ksbar, \ksbar)$ by \cref{thm:representing_U}. Thus the functor $\widetilde \rho \maps 1 \to U(\ksbar)$ admits a unique extension to a morphism of 2-rigs $\overline \rho \maps \ksbar \to \ksbar$.
        \item 
        We have $\TRig(\ksbar, \ksbar) \simeq [\TRig(\ksbar, -), \TRig(\ksbar, -) ]$ by the 2-categorical Yoneda lemma \cite[Ch.\ 8]{2DCats}. Thus  $\overline\rho$ gives a pseudonatural transformation 
        \[
            \overline \rho^* \maps \TRig(\ksbar, -) \To \TRig(\ksbar, -), 
        \]
        the $\R$ component of which, $\overline \rho_\R^*$, maps any morphism of 2-rigs $\phi \maps \ksbar \to \R$ to the composite $\phi \circ \overline \rho$.
        \item     
        \cref{thm:representing_U} gives an equivalence $\TRig(\ksbar, \R) \simeq U(\R)$ mapping $\phi \maps \ksbar \to \R$ to $\phi(1)$. The inverse of this sends $x \in \R$ to $\overline x \maps \ksbar \to \R$. Conjugating by this equivalence gives an equivalence
        \[
            [\TRig(\ksbar, -), \TRig(\ksbar, -) ] \simeq [U, U].
        \]
        and this maps $\overline{\rho}^\ast$ to $F_\rho \in [U, U]$.
    \end{itemize}
   
    We now calculate $F_\rho$ explicitly by seeing how $F_{\rho, R}$ acts on an object $x \in \R$. Passing $x$ through the equivalence $\R \simeq \TRig(\ksbar, \R)$, we obtain
    \[
        \overline x \maps \ksbar \to \R.
    \] 
    Acting on $\overline x$ with $\overline \rho^*$, we obtain
    \[
        \ksbar \xrightarrow{\overline \rho} \ksbar \xrightarrow{\overline x} \R
    \] 
    Turning this back into an object of $U(\R)$ by evaluating at the generator $\eta \maps 1 \to \ksbar$, we obtain
    \[
        1 \xrightarrow{\eta} \ksbar \xrightarrow{\overline \rho} \ksbar \xrightarrow{\overline x} \R, 
    \]
    which is really just
    \[
        1 \xrightarrow{\rho} \ksbar \xrightarrow{\overline x} \R.
    \]
    By the formula for $\overline{x}$ in \cref{lem:U2} we obtain
    \[ 
        \overline x (\rho) = \bigoplus_n \rho(n) \otimes_{k[S_n]} x^{\otimes n}
    \] 
    and by our calculation this is $F_{R, \rho}(x)$.
\end{proof}

A corollary is that by transport of structure across equivalences, the monoidal product on any of the categories 
\[
    \TRig(\ksbar, \ksbar) \simeq [\TRig(\ksbar, -), \TRig(\ksbar, -)] \simeq [U, U], 
\]
given in each case by endofunctor composition, induces a monoidal product on the equivalent categories $\Schur \simeq \Poly \simeq U(\ksbar)$, which we denote by the symbol $\comp$.  This monoidal product is called the \define{substitution product}, or \define{plethysm}.  For a good introduction to plethysm see Macdonald \cite[Appendix IA]{Macdonald}.  In \cref{sec:plethysm} we investigate it in detail and show that it makes the 2-rig $\ksbar$ into a `2-plethory'.  For now, we just state a formula for it:

\begin{cor}
\label{cor:pleth}
    For polynomial species $\rho, \tau \maps \S^{\op} \to \Vect$, the substitution product is given by the formula 
    \[
       \rho \comp \tau = \bigoplus_n \rho(n) \otimes_{k[S_n]} \tau^{\otimes n}
    \]
    and this defines a monoidal product on $\Poly$ whose monoidal unit is $\ksbar(-, 1) \maps \S^{\op} \to \Vect$. 
\end{cor}
\begin{proof}
    In the proof of \cref{thm:schur_versus_poly} where we calculate $F_\rho$, put $\R = \overline{k \S}$, and put $x = \tau$. The composite 
    \[ 
        \ksbar \xrightarrow{\overline \rho} \ksbar \xrightarrow{\overline \tau} \overline{k \S},
    \]
    which is the monoidal product $\overline{\tau} \circ \overline{\rho}$ in $\TRig(\overline{k \S}, \overline{k \S})$, is sent by the functor 
    \[
        \TRig(\overline{k \S}, \overline{k \S}) \to U(\overline{k \S})
    \]
    to the composite 
    \[
        \left(1 \xrightarrow{\eta} \ksbar \xrightarrow{\overline \rho} \ksbar \xrightarrow{\overline \tau} \overline{k \S}\right) = \left(1 \xrightarrow{\rho} \ksbar \xrightarrow{\overline \tau} \overline{k \S}\right)
    \]
    which names the object 
    \[
        \overline \tau (\rho) = \bigoplus_n \rho(n) \otimes_{k[S_n]} \tau^{\otimes n}, 
    \]
    again by \cref{lem:U2}. This completes the proof.
\end{proof}
The transport-of-structure method gives at once both the fact that plethysm on $U(\overline{k \S})$ defines a monoidal product, and that we have monoidal equivalences 
\[
    \left(U(\overline{k \S}), \comp \right) \simeq \left(\TRig(\overline{k \S}, \overline{k \S}), \circ\right) \simeq \left([U, U], \circ\right).
\]

We now use Theorems \ref{thm:schur_versus_poly} and \ref{thm:poly_versus_kS} to transport structure in the other direction, transferring the 2-rig structure on $\ksbar$ to $\Schur$ and $\Poly$. Their structure as linear categories is evident, so the real question is: what do the tensor products on these categories look like? We first answer this question for $\Poly$, and then for $\Schur$.

The category of polynomial species inherits a monoidal structure from $\S$ via Day convolution. In fact it has two, but here we consider the one arising from the \emph{additive} monoidal structure on $\S$, which is given on the level of objects by adding natural numbers, and on the morphism level by group homomorphisms 
\[
    S_i \times S_j \to S_{i+j}
\] 
that juxtapose permutations. These can be linearized to give algebra maps 
\[
    k[S_i] \otimes k[S_j] \to k[S_{i+j}]
\] 
which give the monoidal category structure of $k\S$. This monoidal structure uniquely extends via Day convolution to the Cauchy completion $\ksbar$, which is intermediate between $k\S$ and the category of $\Vect$-valued presheaves on $k\S$. The general formula for the Day convolution product applied to presheaves $\rho, \psi \maps \S\op \to \Vect$ is 
\[
    (\rho \ast \psi)(n) = \bigoplus_{i+j = n} (\rho(i) \otimes \psi(j)) \otimes_{k[S_i \times S_j]} k[S_n] 
\]
By restriction and the isomorphism $\S \cong \S^{\op}$ coming from the fact that $\S$ is a groupoid, this formula gives a tensor product on polynomial species, which Macdonald calls the `induction product' \cite[Appendix IA]{Macdonald}. 

This tensor product is a kind of categorification of the usual definition of product of ordinary polynomials, where given 
\[
    F(x) = \sum_{0 \leq i \leq M} \frac{f_i x^i}{i!} \qquad G(x) = \sum_{0 \leq j \leq N} \frac{g_j x^j}{j!}
\]
the $n$th Taylor coefficient of the product $F(x)G(x)$ is 
\[
    \sum_{i+j = n} \frac{n!}{i! j!} f_i g_j.
\] 

In summary: 
\begin{prop}
\label{prop:poly_2rig}
    The 2-rig $\ksbar$ is equivalent to $\Poly$ made into a 2-rig whose tensor product is given by Day convolution with respect the additive monoidal structure on $\S$.
\end{prop} 

Next we turn to the corresponding 2-rig structure on $\Schur$. 

\begin{prop}
    \label{prop:schur_2rig}
    The 2-rig $\ksbar$ is equivalent to $\Schur$ made into a 2-rig with the pointwise tensor product: given $F, G \in \Schur$ their tensor product $F \otimes G$ has
    \[    (F \otimes G)_\R(x) = F_\R(x) \otimes G_\R(x) \]
    for any 2-rig $\R$ and any object $x \in \R$. This formula also holds for morphisms in $\R$. 
\end{prop}

\begin{proof}  
    We prove this using \cref{thm:schur_versus_poly}. Given $F, G \in \Schur$, up to isomorphism we may assume 
    \[    
        F_\R(x) = \bigoplus_n \rho(n) \otimes_{k[S_n]} x^{\otimes n} 
    \]
    and
    \[   
        G_\R(x) = \bigoplus_n \psi(n) \otimes_{k[S_n]} x^{\otimes n} 
    \]
    for some polynomial species $\rho, \psi$. Then the tensor product of $F$ and $G$ corresponding to the Day tensor product of $\rho$ and $\psi$ is
    \begin{align*}
        (F \otimes G)_\R(x)
        &= \bigoplus_n (\rho \ast \psi)(n) \otimes_{k[S_n]} x^{\otimes n} 
        \\&\cong \bigoplus_n \bigoplus_{i+j = n} (\rho(i) \otimes \psi(j)) \otimes_{k[S_i \times S_j]} k[S_n] \otimes_{k[S_n]} x^{\otimes n} 
        \\&\cong \bigoplus_n \bigoplus_{i+j = n} (\rho(i) \otimes \psi(j)) \otimes_{k[S_i \times S_j]} x^{\otimes n} 
        \\&\cong \big( \bigoplus_i \rho(i)  \otimes_{k[S_i]} x^{\otimes i} \big)
        \otimes \big( \bigoplus_j \psi(j)  \otimes_{k[S_j]} x^{\otimes j} \big) 
        \\&\cong F_\R(x) \otimes G_\R(x) .
    \end{align*}
    The same argument applies to morphisms in $\R$.  
\end{proof}

Using these propositions we can equip $\Schur$ or $\Poly$ with the structure of a 2-rig, making either one into the free 2-rig on one generator. However, in what follows we usually adhere to this discipline: we use $\ksbar$ to stand for the free 2-rig on one generator, and $\Schur \simeq U(\ksbar)$ for the underlying category.

\section{The 2-birig structure on Schur functors}
\label{sec:2-birig}

Now we turn to the 2-birig structure of $\Schur$. It is again helpful to start with a warmup exercise one step down the $n$-categorical ladder. In \cref{sec:intro} we described the `co-operations' on the biring $\Z[x]$. But we did not explain in detail how these co-operations arise. Let us do this now---but in the closely related case of the birig $\N[x]$.

First, recall that a \define{rig} (also called `semirings') is a set $R$ with a commutative monoid structure $(R, +, 0)$ and a commutative monoid structure $(R, \cdot, 1)$ such that $\cdot$ distributes over $+$ and $0 \cdot r = 0$ for all $r \in R$. Commutative monoids admit a tensor product analogous to that of abelian groups, which we denote by $\otimes$.

This leads to a more highbrow definition of a rig: it is a commutative monoid in $(\CMon, \otimes)$   The category of commutative monoid objects in a symmetric monoidal category $(C, \otimes)$ is always cocartesian, with the tensor product $\otimes$ serving to give the coproduct of commutative monoid objects. Thus $(\Rig, \otimes)$ is a cocartesian monoidal category. In other words, $(\Rig\op, \otimes)$ is cartesian monoidal.

Now, a \define{birig} is a rig $B$ such that the functor $\Rig(B, -)$ is equipped with a lift $\Phi_B$ as follows: 
\[
\begin{tikzcd}[column sep = large]
    &
    \Rig
    \arrow[d, "U"]
    \\
    \Rig
    \arrow[ur, "\Phi_B"]
    \arrow[ur, swap, ""{name = phi}, phantom, bend right = 5, pos = 0.6]
    \arrow[r, swap, "{\Rig(B, -)}"]
    \arrow[from = phi, r, Rightarrow, "\sim", sloped]
    &
    \Set
\end{tikzcd}
\] 
where $U$ picks out the underlying set of a rig. However, we can equivalently say that a birig is a rig object in the cartesian monoidal category $(\Rig\op, \otimes)$. Let us explain why. 

Concretely, a rig object in $(\Rig\op, \otimes)$ is just a rig $B$ equipped with these rig homomorphisms:
\begin{itemize}
    \item coaddition: $\alpha \maps B \to B \otimes B$
    \item co-zero: $o \maps B \to \N$
    \item comultiplication: $\mu \maps B \to B \otimes B$
    \item counit, or co-one: $\epsilon \maps B \to \N$
\end{itemize}
obeying dualized versions of the ring axioms. Note that $\N$ appears here because it is the initial rig, hence terminal in $\Rig\op$.

How do we get a rig object in $\Rig\op$ from a rig $B$ for which $\Rig(B, -)$ is equipped with a lift to a functor $\Phi_B \maps \Rig \to \Rig$? We can illustrate this by constructing the coaddition $\alpha \maps B \to B \otimes B$. This comes from \emph{addition}, as follows.

Since addition is defined for every rig and is preserved by rig homomorphisms, it defines a natural transformation from $U \times U$ to $U$:
\[   
\begin{array}{cccl}
   +_R \maps & U(R) \times U(R) &\to& U(R)\\
             & (r, s) & \mapsto & r + s.
\end{array}
\]
By the triangle above this gives a natural transformation 
\[  
    \Rig(B, -) \times \Rig(B, -) \To \Rig(B, -) , 
\]
but the functor at left is naturally isomorphic to $\Rig(B \otimes B, -)$, since $B \otimes B$ is the coproduct in $\Rig$ of two copies of $B$. We thus obtain a natural transformation
\[  
    \Rig(B \otimes B, -) \To \Rig(B, - ) .
\]
By Yoneda, this comes from a rig homomorphism
\[    
    \alpha \maps B \to B \otimes B 
\]
and we define this to be coaddition for the birig $B$. We can similarly construct all the other co-operations that a birig has, and check that they obey dualized versions of the rig laws.

How can we use these ideas to actually compute the birig co-operations for $\N[x]$, the rig of polynomials with natural number coefficients in one variable? Since $\N[x]$ is the free rig on one generator, homomorphisms from it to any other rig correspond to elements of that rig, and it becomes a birig with the identity functor as lift:
\[
\begin{tikzcd}[column sep = large]
    &
    \Rig
    \arrow[d, "U"]
    \\
    \Rig
    \arrow[ur, "1"]
    \arrow[ur, swap, ""{name = phi}, phantom, bend right = 5, pos = 0.6]
    \arrow[r, swap, "{\Rig(\N[x], -)}"]
    \arrow[from = phi, r, Rightarrow, "\sim", sloped]
    &
    \Set
\end{tikzcd}
\] 
Coaddition $\alpha \maps \N[x] \to \N[x] \otimes \N[x]$ is the homomorphism such that precomposing with $\alpha$ gives a function $\alpha^*$ that makes this square commute for any rig $R$:
\[  
\begin{tikzcd}
    \Rig(\N[x] \otimes \N[x], R) 
    \arrow[d, swap, "\sim", sloped]  
    \arrow[r, "\alpha^*"] 
    &
    \Rig(\N[x], R) 
    \arrow[d, "\sim", sloped] 
    \\
    U(R) \times U(R) 
    \arrow[r, swap, "+"] 
    &
    U(R).
    \end{tikzcd}
\]
where $+$ comes from addition on $R$. Let us show that
\[    
    \alpha(x) = x \otimes 1 + 1 \otimes x. 
\]
To prepare for later calculations let us identify $\N[x] \otimes \N[x]$ with $\N[x, y]$ and write
\[     
    \alpha(x) = x + y .
\]
Just as homomorphisms from $\N[x]$ to $R$ correspond to elements of $U(R)$, homomorphisms $f \maps \N[x, y] \to R$ correspond to pairs $(r, s) \in U(R) \times U(R)$ as follows:
\[      
    f(x) = r, \qquad   f(y) = s 
\]
Since
\[   
    (\alpha^* f)(x) = f(x + y) = r + s. 
\]
we see $\alpha^*$ indeed corresponds to addition in $R$ as desired. The same sort of calculation lets us determine all the co-operations on $\N[x]$:
\begin{itemize}
    \item coaddition: $\alpha(x) = x + y \in \N[x, y]$
    \item co-zero: $o(x) = 0 \in \N$
    \item comultiplication: $\mu(x) = xy \in \N[x, y]$
    \item co-one: $\epsilon(x) = 1 \in \N$.
\end{itemize}

With our warmup exercise complete, we can now copy this reasoning to show that $\Schur$ is a 2-birig and compute some of its co-operations. First:

\begin{defn} 
    A \define{2-birig} is a 2-rig $\B$ for which the functor $\TRig(\B, -)$ is equipped with a lift to a functor $\Phi_\B \maps \TRig \to \TRig$:
    \[
    \begin{tikzcd}[column sep = huge]
        &
        \TRig
        \arrow[d, "U"]
        \\
        \TRig
        \arrow[ur, "\Phi_\B"]
        \arrow[ur, swap, ""{name = phi}, phantom, bend right = 5, pos = 0.6]
        \arrow[r, swap, "{\TRig(\B, -)}"]
        \arrow[from = phi, r, Rightarrow, "\sim", sloped]
        &
        \CCat.
    \end{tikzcd}
    \] 
\end{defn}

To describe the 2-birig structure on $\Schur$, recall that we have made $\Schur$ into a 2-rig equivalent to $\ksbar$. So, we start by putting a 2-birig structure on $\ksbar$. By \cref{thm:representing_U} we know that $\ksbar$ represents the 2-functor $U$. This makes $\ksbar$ into a 2-birig with the identity lift:
\[
\begin{tikzcd}[column sep = huge]
    &
    \TRig
    \arrow[d, "U"]
    \\
    \TRig
    \arrow[ur, "1"]
    \arrow[ur, swap, ""{name = phi}, phantom, bend right = 5, pos = 0.6]
    \arrow[r, swap, "{\TRig(\ksbar, -)}"]
    \arrow[from = phi, r, Rightarrow, "\sim", sloped]
    &
    \CCat.
\end{tikzcd}
\] 

From this we can compute co-operations on $\ksbar$, just as we did in our warm-up exercise. For this we need to show that just as $\Rig\op$ is a cartesian category, $\TRig\op$ is a cartesian 2-category.

We saw that $(\Rig\op, \otimes)$ is cartesian monoidal by noting that the category of commutative monoids in any symmetric monoidal category is cocartesian monoidal, and
rigs are commutative monoid objects in $(\CMon, \otimes)$. We now categorify this argument replacing $\CMon$ with $\Cauch\Lin$ (see \cref{defn:CauchLin}). In the proof of \cref{lem:U2} we saw that $(\Cauch\Lin, \boxtimes)$ is a symmetric monoidal 2-category. The unit object for this symmetric monoidal 2-category is $\Fin\Vect$.

A 2-rig $\R$ is a symmetric monoidal linear category that is also Cauchy complete. Thus, it comes with a tensor product or multiplication
\[ 
    m \maps \R \boxtimes \R \to \R
\]
which is a morphism in $\Cauch\Lin$. It also comes with a unit object $I \in \R$, which determines a morphism 
\[    
    I \maps \Fin\Vect \to \R  
\]
in $\Cauch\Lin$, unique up to natural isomorphism, such that $i(k) = I$. It also comes with an associator, left and right unitors, and symmetry that are 2-morphisms in $\Cauch\Lin$, obeying the usual equations in the definition of symmetric monoidal category. We may summarize all this by saying that a 2-rig is a symmetric pseudomonoid in the symmetric monoidal 2-category $(\Cauch\Lin, \boxtimes)$. 
 
Given 2-rigs $\R$ and $\R'$ there is a natural way to make $\R \boxtimes \R'$ into a 2-rig. The multiplication in $\R \boxtimes \R'$ is the composite
\[   (\R \boxtimes \R') \boxtimes (\R \boxtimes \R') \xrightarrow{1 \boxtimes S_{\R', \R} \boxtimes 1} (\R \boxtimes \R) \boxtimes (\R' \boxtimes \R') \xrightarrow{m \boxtimes m'} \R \boxtimes \R' \]
where $m$ is the multiplication for $\R$, $m'$ is the multiplication for $\R'$, $S_{\R, \R'}$ is the symmetry in $(\Cauch\Lin, \boxtimes)$, and we have suppressed associators. The unit for $\R \boxtimes \R'$ is 
\[   \Fin\Vect \xrightarrow{\sim} \Fin\Vect \boxtimes \Fin\Vect \xrightarrow{I \boxtimes I'} \R \boxtimes \R' \]
where $I$ is the unit for $\R$ and $I'$ is the unit for $\R'$. The rest of the 2-rig structure is equally straightforward. The interesting fact is that this tensor product of 2-rigs is their coproduct:

\begin{lem}
\label{lem:2Rig-cocartesian}
    The symmetric monoidal 2-category $\TRig$ is cocartesian, with the coproduct of $\R, \R' \in \TRig\op$ being $\R \boxtimes \R'$.
\end{lem}

\begin{proof}
This follows from a general result proved by Sch\"appi \cite[Thm.\ 5.2]{Schappi}: given symmetric pseudomonoids $\M$ and $\M'$ in a symmetric monoidal bicategory $(\C, \otimes)$, there is a natural way to make $\M \otimes \M'$ into a symmetric pseudomonoid, of which the tensor product of 2-rigs is an example. Furthermore there is a 2-category of symmetric pseudomonoids in $(\C, \otimes)$, and this is cocartesian in the 2-categorical sense, with $\M \otimes \M'$ being the 2-categorical coproduct of $\M$ and $\M'$.
\end{proof}

We now compute the coaddition 
\[   
    \alpha \maps \ksbar \to \ksbar \boxtimes \ksbar. 
\]  
For this, it is helpful to note that, just as the ring $\Z[x] \otimes \Z[x]$ is isomorphic to $\Z[x, y]$, $\ksbar \boxtimes \ksbar$ is equivalent to $\overline{k\SS(2)}$, where $\SS$ is the left adjoint 2-functor sending categories to symmetric monoidal 2-categories, as in Lemma \ref{lem:freeSMC}.  Indeed:

\begin{lem}
\label{lem:n-variable}
    The following are equivalent Cauchy complete linear categories:
    \begin{itemize}
        \item $(\ksbar)^{\boxtimes n}$ \vskip 0.5em
        \item $\overline{k\SS(n)}$ \vskip 0.5em
        \item the linear category of \define{$n$-variable polynomial species}: functors $\S^n \to \Fin\Vect$ such that $F(x) \cong \{0\}$ except for finitely many isomorphism classes of objects $x$, and natural transformations between these. 
    \end{itemize}
    The underlying category of any of these is equivalent to the category of \define{$n$-variable Schur functors}, $[U^n, U]$.
\end{lem}

\begin{proof}
    First we show that $(\ksbar)^{\boxtimes n} \simeq \overline{k\SS(n)}$. Since $\SS$ is a left adjoint, $\SS(n)$ is the $n$-fold coproduct of $\SS(1) \simeq \S$ in $\SMC$. Since $\SMC$ has biproducts in the 2-categorical sense \cite[Thm.\ 2.3]{FongSpivak}, this $n$-fold coproduct is equivalent to $\S^n$. We thus have these equivalences in $\Cauch\Lin$:
    \[   \overline{k\SS(n)} \simeq \overline{k(\S^n)} \simeq 
    \overline{(k\S)^{\boxtimes n}} \simeq (\ksbar)^{\boxtimes n} .\]
    Here the second uses the fact that $k(-) \maps \SMC$ sends products of categories to tensor products of linear categories (cf.\ the proof of \cref{lem:U1}). The third uses the fact that $\overline{(-)}$ preserves the tensor product of linear categories (cf.\ the proof of \cref{lem:U2}). 
    
    The equivalence of $\overline{k\SS(n)}$ to the linear category of $n$-variable polynomial species can be shown using the same style of argument as in the proof of \cref{thm:poly_versus_kS}.
    
    Finally, note that 
    \[
        [U^n, U] \simeq 
        [\TRig(\ksbar, -)^n, 
        \TRig(\ksbar, -) ] \simeq
        \TRig(\ksbar, \ksbar^{\boxtimes n}) \simeq U(\ksbar^{\boxtimes n}).
    \] 
    Here the first equivalence comes from  \cref{thm:representing_U}, the second comes from the 2-categorical Yoneda lemma and the fact that $(\ksbar)^{\boxtimes n}$ is the $n$-fold coproduct of $\ksbar$, and the third is \cref{thm:poly_versus_kS}. 
    
    We can trace through the equivalences to make explicit the equivalence 
    \[ \phi \maps U(\ksbar^{\boxtimes n}) \simeq [U^n, U]. \]
    An object of $U(\ksbar^{\boxtimes n})$ is, by definition, an object of the linear Cauchy completion of $(k \S)^{\otimes n} \simeq k(\S^n) \simeq k(\SS(n))$. It is given by a collection $F$ of finite-dimensional representations of $n$-fold products of symmetric groups, i.e.\ functors 
    \[ F(m_1, \ldots, m_n) \maps S_{m_1} \times \cdots \times S_{m_n} \to \Vect\] 
    that are zero-dimensional except for finitely many tuples $(m_1, \ldots, m_n)$ of natural numbers. 
    
    For $\R$ a 2-rig, given a tuple of objects $x = (x_1, \ldots, x_n) \in U^n(\R)$, there is a corresponding symmetric monoidal functor $\S^n \to \R$ that is uniquely induced by the tuple $x$. It takes the object $\mathbf{m} = (m_1, \ldots, m_n) \in \S^n$ to $x^{\mathbf{m}} = x^{\otimes m_1} \otimes \cdots \otimes x^{\otimes m_n}$. The group $S_{m_1} \times \cdots \times S_{m_n}$ acts on this object by permuting tensor factors, and symmetric algebras $kS_{m_1} \otimes \cdots \otimes kS_{m_n}$ act linearly on such objects. In other words, from the tuple $x$ we derive a symmetric monoidal linear functor
    \[ (k\S)^{\otimes n} \to \R \]
    uniquely up to isomorphism. Finally, passing up to the Cauchy completions, we obtain a single object of $\R$ defined by the formula 
    \[ \phi(F)(x_1, \ldots, x_n) = \bigoplus_\mathbf{m = (m_1, \ldots, m_n)} F(m_1, \ldots, m_n) \otimes_{k(S_{m_1}) \otimes \dots \otimes k(s_{m_n})} (x^{\otimes m_1} \otimes \cdots \otimes x^{\otimes m_n}), \] 
    which, adapting a familiar multi-index notation $\mathbf{m} = (m_1, \ldots, m_n)$, might be be more neatly written as 
    \[ \phi(F)(x) = \int^{\mathbf{m} \in \S^n} F(\mathbf{m}) \cdot x^{\otimes \mathbf{m}} = \bigoplus_\mathbf{m} F(\mathbf{m}) \otimes_{kS_{\mathbf{m}}} x^{\otimes \mathbf{m}}. \qedhere \]
\end{proof}

It follows that the functor $U \times U$ is represented by $\ksbar \boxtimes \ksbar \simeq \overline{k\SS(2)}$. Coaddition can then be defined to be the morphism of 2-rigs 
\[  
    \alpha \maps \ksbar \to \overline{k\SS(2)} 
\]
such that precomposition with this gives a functor $\alpha^*$ for which this square commutes up to a natural isomorphism:
\[  
\begin{tikzcd}
    \TRig(\overline{k\SS(2)}, \R) 
    \arrow[d, "\sim", sloped, swap]  
    \arrow[r, "\alpha^*"] 
    &
    \TRig(\ksbar, \R) 
    \arrow[d, "\sim", sloped] 
    \\
    U(\R) \times U(\R) 
    \arrow[ur, Rightarrow, "\sim", sloped]
    \arrow[r, swap, "\oplus"] 
    &
    U(\R).
\end{tikzcd}
\]
We can follow the proof of the Yoneda lemma to determine what $\alpha$ must be. Put $\R = \overline{k\SS(2)}$, and let $x$ denote the object generating $\ksbar$, and $x, y$ the generators of $\overline{k\SS(2)}$. Chasing the identity object 
\[
    1 \in \TRig(\overline{k\SS(2)}, \overline{k\SS(2)})
\]
around the square, the left vertical functor takes $1$ to the pair $(x, y)$; applying the bottom horizontal map, we arrive at $x \oplus y$. On the other hand, $\alpha^\ast(1) = 1 \circ \alpha = \alpha$, and this maps down to the value $\alpha(x)$, which as we just saw must match $x \oplus y$. Hence 
\[
    \alpha \maps \ksbar \to \overline{k\SS(2)}
\]
must be the 2-rig map, unique up to isomorphism, such that $\alpha(x) = x \oplus y$ in $\overline{k\SS(2)}$.

We can determine other co-operations on $\ksbar$ in the same way:
\begin{itemize}
    \item coaddition: $\alpha(x) = x \oplus y \in \overline{k\SS(2)}$
    \item cozero: $o(x) = 0 \in \Fin\Vect$
    \item comultiplication: $\mu(x) = x \otimes y \in \overline{k\SS(2)}$
    \item counit: $\epsilon(x) = 1 \in \Fin\Vect$.
\end{itemize}

Next we interpret the 2-birig structure of $\ksbar$ directly in terms of $\Schur$. Recall that having an object $F \in \Schur$ is the same as having an endofunctor $F_\R \maps U(\R) \to U(\R)$ on the underlying category of every 2-rig $\R$, depending pseudonaturally on $\R$. This is a technical statement, but we may think of it as saying that a Schur functor $F \maps U \to U$ is a unary functorial operation that is definable in the language of 2-rigs. Similarly, a 2-variable Schur object is the same as an endofunctor $F_\R \maps U(\R) \times U(\R) \to U(\R)$ that depends pseudonaturally on $\R$, or a binary functorial operation definable in the language of 2-rigs. Our next result gives straightforward interpretations of operations and co-operations on $\overline{k \S}$ in terms of their effect on Schur objects.

\begin{thm}
\label{thm:schur_2-birig}
    Under the equivalence $U(\overline{k \S}) \simeq [U, U]$, the 2-rig operations on $U(\overline{k \S})$ correspond to 2-rig operations on $[U, U]$ as follows: 
    \begin{itemize}
        \item addition: Coproduct $\oplus \maps U(\overline{k \S}) \times U(\overline{k \S}) \to U(\overline{k \S})$ corresponds to pointwise coproduct, where $F \oplus G \in \Schur$ is given by
        \[   (F \oplus G)_\R(r)   = F_\R(r) \oplus G_\R(r) \]
        where $\oplus$ on the right is the biproduct in $\R$.
        \item `zero', or additive unit: the initial object $0$ in $U(\overline{k \S})$ corresponds to pointwise $0 \in \Schur$ given by
        \[    0_\R(r) = 0  \]
        where $0$ at right is the zero object in $\R$.
        \item multiplication: The convolution product $\ast \maps U(\overline{k \S}) \times U(\overline{k \S}) \to U(\overline{k \S})$ corresponds to pointwise tensor: $F \otimes G \in \Schur$ is given by
        \[    (F \otimes G)_\R(r) = F_\R(r) \otimes G_\R(r) .\]
        \item `one', or multiplicative unit: The convolution unit $I$ of $U(\overline{k \S})$ corresponds to the pointwise monoidal unit $I \in \Schur$ given by
        \[    I_\R(r) = I  \]
        where $I$ at right is the unit for the tensor product in $\R$.
    \end{itemize}
    Under the equivalences $U(\overline{k \S}^{\boxtimes n}) \simeq [U^n, U]$, co-operations on the 2-rig $\overline{k \S}$ correspond to precomposition with operations as follows:
    \begin{itemize}
        \item coaddition: $U(\alpha) \maps U(\overline{k \S}) \to U(\overline{k \S}^{\boxtimes 2})$ corresponds to precomposition $[\oplus, U]\colon [U, U] \to [U^2, U]$, taking $F \maps U \to U$ to $F \circ \oplus \maps U^2 \to U$ given by
        \[   
            (F \circ \oplus)_{\R}(r, s) = F_\R(r \oplus s).
        \]
        \item `co-zero', or coadditive counit: $U(o) \maps U(\overline{k \S}) \to \Fin\Vect$ corresponds to precomposition with, i.e.\ evaluation at, $0$: the corresponding map $[U, U] \to [1, U]$ is given by
        \[    
            (F \circ \lceil 0 \rceil)_\R = F_\R(0).
        \]
        \item comultiplication: $U(\mu) \colon U(\overline{k \S}) \to U(\overline{(k \S)^{\boxtimes 2}})$ corresponds to precomposition $[\otimes, U] \colon [U, U] \to [U^2, U]$, taking $F \maps U \to U$ to $F \circ \otimes \maps U^2 \to U$ given by  
        \[  
            (F \circ \otimes)_{\R}(r, s) = F_\R(r \otimes s). 
        \]
        \item `co-one' or comultiplicative counit: $U(\epsilon) \maps U(\overline{k \S}) \to \Fin\Vect$ corresponds to precomposition with, i.e.\ evaluation at, the monoidal unit $I$: the corresponding map $[U, U] \to [1, U]$ is given by 
        \[    
            (F \circ \lceil I \rceil)_\R = F_\R(I).
        \]
    \end{itemize}
\end{thm}
\begin{proof}
That the convolution product $\ast$ corresponds to pointwise tensor on $\Schur$ was the content of \cref{prop:schur_2rig}, and the other operations are evident. 

On the other hand, as we saw in the case of coaddition using the Yoneda lemma, each co-operation cited in the theorem, of the form
\[   
    \beta \maps \ksbar \to \ksbar^{\boxtimes n}\simeq \overline{k\SS(n)}, 
\]
is \emph{defined} to be the 2-rig map, uniquely determined up to isomorphism, such that 
\[
    \beta(x) = b(x_1, \ldots, x_n)
\]
where $x_1, \ldots, x_n$ are the generators of $\overline{k\SS(n)}$, and $b \maps U(\overline{k\SS(n)})^n \longrightarrow  U(\overline{k\SS(n)})$ is the corresponding Schur-functor operation evaluated at $\R = \overline{k\SS(n)}$. Finally, according to these definitions, we have a square 
\[  
\begin{tikzcd}
    {[U, U]}
    \arrow[d, "\sim" {anchor=south, rotate=90, inner sep=.5mm}, swap]  
    \arrow[r, "{[b, U]}"] 
    &
    {[U^n, U]}
    \arrow[d, "\sim" {anchor=north, rotate=90, inner sep=.5mm}] 
    \\
     U(\ksbar) 
    \arrow[ur, Rightarrow, "\sim" {rotate=20, inner sep=1mm}]
    \arrow[r, swap, "U(\beta)"] 
    &
    U(\overline{k\SS(n)})
\end{tikzcd}
 \]
that commutes up to a 2-cell isomorphism, using \cref{thm:schur_versus_poly}. Chasing an object $F \in [U, U]$ both ways around the square, this says precisely that the co-operation $\beta$ corresponds to the assignment $F \mapsto F \circ b$, as stated for each case in the theorem. \end{proof}

\section{Plethories, plethysm, and 2-plethories}
\label{sec:plethysm}

We have defined a birig to be a rig $B$ together with a lifting $\Phi_B$ of the representable functor $\Rig(B, -) \maps \Rig \to \Set$ through the forgetful functor $U \maps \Rig \to \Set$.  We have an analogous notion of 2-birig. We now present another viewpoint on these notions, which paves the way for a simple definition of rig-plethory, and then of 2-plethory: that a birig is simply an endofunctor
\[
    \Phi \maps \Rig \to \Rig
\]
that is a right adjoint. None of this depends on special features of $\Rig$, beyond the fact that the forgetful functor $U \maps \Rig \to \Set$ is monadic. Thus, we might as well work more generally.

\subsection{$M$-plethories}
\label{sec:M-plethories}

For this section, let $M$ be a monad on $\Set$ and $M\Alg$ its Eilenberg--Moore category, with free-forgetful adjunction as below.
\[
\begin{tikzcd}
    \Set
    \arrow[r, bend left, "F", pos = 0.55]
    \arrow[r, phantom, "\bot", pos = 0.6]
    &
    M\Alg
    \arrow[l, bend left, "U", pos = 0.45]
\end{tikzcd}\]
In order to define $M$-plethories we introduce $M$-bialgebras. The two monads to keep in mind are those for rings and rigs. When $M$ is the monad for rings an $M$-bialgebra will be a biring, and when $M$ is the monad for rigs an $M$-bialgebra will be a birig.

\begin{defn}
    An \define{$M$-bialgebra} is an $M$-algebra $B$ equipped with a lift $\Phi_B$ of the functor $M\Alg(B,-) \maps M\Alg \to \Set$ through $U$:
    \[
    \begin{tikzcd}[column sep = large]
        &
        M\Alg
        \arrow[d, "U"]
        \\
        M\Alg
        \arrow[ur, "\Phi_B"]
        \arrow[ur, swap, ""{name = phi}, phantom, bend right = 5, pos = 0.6]
        \arrow[r, swap, "{M\Alg(B, -)}"]
        \arrow[from = phi, r, Rightarrow, "\sim", sloped]
        &
        \Set.
    \end{tikzcd}
    \] 
\end{defn}

Two other equivalent definitions of $M$-bialgebra will often be more useful. For these we need some lemmas. 

\begin{lem}
\label{lem:rightrepable}
    If a functor $G \maps \C \to \Set$ is a right adjoint, then it is representable. Moreover, $G$ is represented by $F(1)$, where $F$ is left adjoint to $G$.
\end{lem}

\begin{proof}
    For any such functor $G \maps \C \to \Set$ with a left adjoint $F$, we have natural isomorphisms 
    \[
        G \cong \Set(1, G(-)) \cong \C(F(1), -)
    \]
    (where the first exists for any functor with codomain $\Set$), so that $F(1)$ is a representing object for $G$. 
\end{proof}

\begin{lem} 
\label{lem:repr}
    A right adjoint endofunctor $\Phi \maps M\Alg \to M\Alg$ determines, uniquely up to isomorphism, an $M$-algebra $B$ which carries an $M$-bialgebra structure with lift given by $\Phi$. 
\end{lem}

\begin{proof}
    The composite $U\Phi$ is a right adjoint. \cref{lem:rightrepable} then tells us that there is an $M$-algebra $B$ such that $U\Phi \cong M\Alg(B,-)$, which is precisely the condition of $\Phi$ lifting $M\Alg(B,-)$ through $U$.
\end{proof}

\begin{thm}
\label{thm:liftadjoint}
    Any lift $\Phi_B \maps M\Alg \to M\Alg$ of a representable $M\Alg(B, -) \maps M\Alg \to \Set$ is a right adjoint. 
\end{thm}

\begin{proof}
    The proof uses some well-known facts:
    \begin{enumerate}
        \item Every representable functor $M\Alg(B, -) \maps M\Alg \to \Set$ has a left adjoint, denoted $B \cdot -$. This takes a set $X$ to the coproduct $B \cdot X$ in $M\Alg$ of an $X$-indexed collection of copies of $B$. 
        \item A lift $\Phi_B$ of $M\Alg(B, -)$ through $U$ is precisely equivalent to a left $M$-algebra structure on the representable $M\Alg(B, -)$: that is, a natural transformation
        \[
        \theta \maps M \circ M\Alg(B, -) \to M\Alg(B, -)
        \]
        obeying the usual axioms for $M$-algebras (this fact is the 2-universal property of the Eilenberg--Moore construction). 
        \item Giving such an $M$-algebra structure $\theta$ is equivalent to giving a morphism of monads 
        \[
        M \to M\Alg(B, -) \circ (B \cdot -) = M\Alg(B, B \cdot -)
        \]
        and it is also equivalent to giving a \emph{right} $M$-algebra structure 
        \[
        \xi \maps (B \cdot -) \circ M \to (B \cdot -)
        \]
        which may also be written as $\xi \maps B \cdot M- \to B \cdot -$. The counit $\varepsilon$ of the adjunction $(B \cdot -) \dashv M\Alg(B, -)$ coequalizes the following parallel pair:
        \[
        \begin{tikzcd}
            B \cdot M \circ M\Alg(B, -)  
            \arrow[r, shift left=0.5ex, "{\xi M\Alg(B, -)}"] 
            \arrow[r, shift right=0.5ex, swap, "B \cdot \theta"] 
            & [3em] 
            B \cdot M\Alg(B, -) 
            \arrow[r, rightarrow, "\epsilon"] 
            & 
            \mathrm{Id} 
        \end{tikzcd}
        \]
        \item $M\Alg$ is cocomplete \cite[Theorem 4.3.5]{BorceuxII}, and in particular has reflexive coequalizers.
    \end{enumerate}
    
    Now, the existence of a left adjoint $\Psi_B$ of $\Phi_B$ may be exhibited directly as follows. It suffices to show that for any $M$-algebra $R$, the functor $M\Alg(R, \Phi_B-) \maps M\Alg \to \Set$ is representable; the representing object $S$ is then the value $\Psi_B R$. Let $\alpha \maps M UR \to UR$ denote the $M$-algebra structure on $UR$ afforded by the structure of $R$ as an $M$-algebra. We form $S$ as the coequalizer of the following reflexive pair:
    \[
    \begin{tikzcd}
        B \cdot MUR  
        \arrow[r, rightarrow, shift left=0.5ex, "\xi UR"] 
        \arrow[r, rightarrow, shift right=0.5ex, swap, "B \cdot \alpha"] 
        & 
        B \cdot UR.
    \end{tikzcd}
    \]
    (This coequalizer might be suggestively written as $(B \cdot -) \circ_M UR$, recalling the construction of a tensor product of a right $M$-module $(B \cdot -)$ with a left $M$-module $UR$.) The rest follows the proof of the adjoint lifting theorem \cite[Theorem 4.5.6 and Exercise 4.8.6]{BorceuxII}, although we have given enough hints that the patient reader could work out the details.
\end{proof} 

We defined an $M$-bialgebra to be an $M$-algebra $B$ equipped with a lift $\Phi_B$ of the functor $M\Alg(B,-)$ it represents. \cref{thm:liftadjoint} then tells us that $\Phi_B$ is necessarily a right adjoint. Conversely, \cref{lem:repr} tells us that any right adjoint $M\Alg \to M\Alg$ is a lift of a representable, and thus determines an $M$-bialgebra. If one takes morphisms of $M$-bialgebras as pointing in the same direction as morphisms of their underlying $M$-algebras $B$ (hence in the direction opposite to natural transformations between representable functors $M\Alg(B, -)$, where $B$ is in the contravariant argument), then that direction is in alignment with identifying $M$-bialgebras with left adjoint endofunctors on $M\Alg$. Thus, with this direction of morphisms of $M$-bialgebras, we have as a corollary of \cref{lem:repr} and \cref{thm:liftadjoint} the following result.  

\begin{cor}
\label{cor:leftadjoint}
We have an equivalence of categories
\[
    M\Bialg \simeq \LAdj(M\Alg, M\Alg).
\] 
\end{cor}

The category $\LAdj(M\Alg, M\Alg)$ is a monoidal category where the monoidal product is endofunctor composition. This monoidal product transports across the equivalence, making $M\Bialg$ into monoidal category. We denote the monoidal product of two $M$-bialgebras $B$ and $C$ by $B \odot C$. The identity functor on $M\Alg$ is the unit for endofunctor composition, and by \cref{lem:rightrepable} this identity functor corresponds to the $M$-bialgebra $F(1)$, so the monoidal unit for $\odot$ must be $F(1)$. 

Given an $M$-bialgebra $B$, we let $B_o$ denote its underlying $M$-algebra, and temporarily write $F(1)$ to mean the above $M$-bialgebra. Write $\Psi_B$ for the left adjoint of the endofunctor lift $\Phi_B$. 

\begin{lem}
   There are natural isomorphisms 
    \[  B_o \cong (B \odot  F(1))_o \cong \Psi_B(F(1)_o) .\]
\end{lem}

\begin{proof}
\label{lem:isomorphisms}
    The mate of the natural isomorphism
    \[U \circ \Phi_B \cong M\Alg(B_o, -)\]
    is a natural isomorphism 
    \[\Psi_B \circ F \cong B_o \cdot -\]
    where $U \cong M\Alg( F(1), -)$ has left adjoint $F= F(1) \cdot -$. Applying
    \[\Psi_B \circ ( F(1)_o \cdot -) \cong B_o \cdot - \]
    to the terminal set $1$ we obtain
    \[\Psi_B( F(1)_o) \cong B_o.\qedhere\]
\end{proof}

\begin{prop}
\label{prop:odot}
    Let $B$ and $C$ be $M$-bialgebras. Then there is a natural $M$-algebra isomorphism 
    \[ \alpha_{B,C} \maps (B \odot C)_o \xrightarrow{\,\sim\,} \Psi_B(C_o) .\]
\end{prop}
\begin{proof}
    This isomorphism is the composite of natural isomorphisms we have seen:
    \[
        (B \odot C)_o \cong \Psi_{B \odot C}( F(1)) \cong \Psi_B \Psi_C( F(1)) \cong \Psi_B(C_o) .\qedhere 
    \]
\end{proof}

With this proposition in hand, for an $M$-bialgebra $B$ there is no real harm in using $B \odot C$ both for the monoidal product (if $C$ is an $M$-bialgebra) and for the value of the left adjoint $\Psi_B(C)$ (if $C$ is merely an $M$-algebra), letting the meaning be inferred from context. In fact, what we have is an actegory structure whereby the monoidal category $M\Bialg$ acts on $M\Alg$ via a functor
\[
\odot \maps M\Bialg \times M\Alg \to M\Alg, 
\] 
so that if $B$ and $C$ are $M$-bialgebras and $R$ is an $M$-algebra, then there is a coherent compatibility constraint 
\[
(B \odot C) \odot R \cong B \odot (C \odot R)
\]
which is just another way of restating the way in which the monoidal product $\odot$ on $M\Bialg$ was defined by transport of structure: 
\[ 
\Psi_{B \odot C} \cong \Psi_B \circ \Psi_C. 
\]
Summarizing \cref{lem:repr}, \cref{thm:liftadjoint}, \cref{cor:leftadjoint}, and \cref{prop:odot}, we have:

\begin{cor}
\label{thm:MBialg}
    If $M$ is a monad on $\Set$, then the following are equivalent as monoidal categories:
    \begin{enumerate}
        \item $(M\Bialg, \odot)$
        \item $(\LAdj(M\Alg, M\Alg), \circ)$
        \item $(\RAdj(M\Alg,M\Alg), \circ)\op$.
    \end{enumerate}
\end{cor}

Let us give a more explicit description of $B \odot R \cong \Psi_B(R)$ in the special case where $M$ is the monad for commutative rings. This construction can be extracted from the proof of \cref{thm:liftadjoint} and appears in a number of places in the literature, e.g., Borger--Wieland \cite{Plethystic} and  Tall--Wraith \cite{TallWraith}, so we simply state the result here. 

\begin{prop}
\label{prop:odotexplicit} 
For a biring $B$, introduce Sweedler notation for the comultiplication and coaddition as follows:
\[
\Delta_B^{+}(b) = \sum_i b_i^{(1)} \otimes b_i^{(2)}, \qquad
\Delta_B^{\times}(b) = \sum_i b_i^{[1]} \otimes b_i^{[2]} .
\]
Then for any ring $R$, $\Psi_B(R)$ is the ring with generators $b \odot r$ subject to the following relations: 
\[
bb' \odot r = (b \odot r)(b' \odot r) \qquad (b + b') \odot r = (b \odot r) + (b' \odot r), 
\]
\[b \odot (r+r') = \Delta_B^{+}(r, r') \coloneqq \sum_i (b_i^{(1)} \odot r)(b_i^{(2)} \odot r'),
\]
\[ 
b \odot (rr') = \Delta_B^{\times}(r, r') \coloneqq \sum_i (b_i^{[1]} \odot r)(b_i^{[2]} \odot r'), \]
\[0 \odot r = 0, \qquad 1 \odot r = 1, \qquad b \odot -r = \nu(b) \odot r\]
\[b \odot 1 = i(b), \qquad b \odot 0 = o(b).\]
\end{prop}

We now define plethories for an arbitrary monad $M$:

\begin{defn}
\label{def:Mplethory}
    An \define{$M$-plethory} is a monoid in the monoidal category $(M\Bialg, \odot)$.
\end{defn} 

When $M$ is the monad for rings, Borger and Wieland \cite{Plethystic} call an $M$-plethory just a `plethory', but we shall use the term \define{ring-plethory}. We are also interested in the case where $M$ is the monad for rigs; then we call an $M$-plethory a \define{rig-plethory}. The simplest $M$-plethory of all is given by the identity functor on $M\Alg$ viewed as right adjoint comonad. The underlying $M$-algebra of this $M$-plethory is $F(1)$. When $M$ is the monad for rings this gives the ring-plethory structure on $\Z[x]$. When $M$ is the monad for rigs it gives the rig-plethory structure on $\N[x]$. 

By \cref{thm:MBialg} we can also think of an $M$-plethory $B$ as a left adjoint monad $\Psi_B \maps M\Alg \to M\Alg$ or as a right adjoint comonad $\Phi_B \maps M\Alg \to M\Alg$, though in the latter case we need to remember that a morphism of $M$-plethories goes in the \emph{opposite} direction from a comonad morphism. We thus have three perspectives on $M$-plethories. 
For example, in \cref{thm:lambda_plethory} we describe the ring-plethory structure on the ring of symmetric functions, $\Lambda$. In the first perspective, we think of this ring-plethory as a monoid in $(\Biring, \odot)$. In the second, we think of it as a left adjoint monad $\Psi_\Lambda \maps \Ring \to \Ring$. Here $\Psi_\Lambda(R) \cong \Lambda \odot R$ is called the `free $\lambda$-ring' on $R$. In the third, we think of this ring-plethory as a right adjoint comonad $\Phi_\Lambda \maps \Ring \to \Ring$. Here $\Phi_\Lambda(R)$ is called the ring of `Witt vectors' of $R$.

To actually obtain the ring-plethory structure on $\Lambda$, we need yet a fourth perspective on $M$-plethories, one which makes more explicit contact with the operations of $M$. We develop this fourth perspective now. 

As we have noted, an $M$-bialgebra structure is equivalent to an $M$-algebra $B$ together with a lift $\Phi$ of $M\Alg(B, -) \maps M\Alg \to \Set$ through the forgetful functor $U \maps M\Alg \to \Set$. Thus $\Phi(C)$ consists of a hom-set $\hom(B, C)$ endowed with an $M$-algebra structure induced from the bialgebra structure on $B$. In this way, for $M$-algebras $A, C$ we may speak of $M$-algebra maps $A \to \hom(B, C)$; in particular, it makes sense to speak of $M$-algebra maps $B \to \hom(B, B)$, formalized by maps $h \maps B \to \Phi(B)$. At the underlying set level, a function $B \to \hom(B, B)$ corresponds to a binary operation $B \times B \to B$.

The guiding idea behind the fourth perspective is to see an $M$-plethory as an $M$-bialgebra $B$ together with an $M$-algebra map $h \maps B \to \Phi(B)$ whose corresponding binary operation $B \times B \to B$, called \define{plethysm}, is associative and has a unit $e \maps 1 \to B$. However, the situation is slightly richer than that: we need $h \maps B \to \Phi(B)$ to correspond to a \emph{bialgebra} map $m \maps B \odot B \to B$. We also need the unit $e \maps 1 \to B$ to correspond to a bialgebra map $F(1) \to B$. As we shall see, an $M$-plethory is precisely equivalent to an algebra map $h \maps B \to \Phi(B)$ and unit satisfying these conditions. (In general $\Phi(B)$ does not carry $M$-bialgebra structure, so we cannot simply ask that $h \maps B \to \Phi(B)$ be a bialgebra map.)

We shall go further, though, by writing out these conditions explicitly in terms of co-operations. In general, an $M$-algebra can be described as a set $A$ equipped with a collection of $J$-ary operations
\[    \theta \maps A^J \to A \]
obeying certain equations. Here $J$ potentially ranges over all sets, though finite sets suffice if the monad $M$ comes from a Lawvere theory, as in our main examples of interest. Similarly, an $M$-bialgebra can be described as an $M$-algebra $B$ equipped with $J$-ary `co-operations' obeying the same sort of equations. These co-operations are $M$-algebra morphisms
\[    [\theta] \maps B \to J \cdot B \]
where $J \cdot B$ is the coproduct in $M\Alg$ of $J$ copies of $B$. In \cref{sec:intro} we described some of these co-operations in the case of the biring $\Z[x]$, and at the start of \cref{sec:2-birig} we worked them out in a principled way for the birig $\N[x]$. We now describe how an $M$-plethory structure interacts with these co-operations. 

Let us take as our starting point that an $M$-plethory consists of a bialgebra $(B, \Phi)$ together with a comonad structure 
\[
    \delta \maps \Phi \To \Phi\Phi, \qquad \varepsilon \maps \Phi \To 1. 
\]
These maps correspond to bialgebra morphisms $B \odot B \to B$ and $F(1) \to B$ between the representing objects. 

\begin{lem}
\label{lem:M-plethory_h}
    A natural transformation $\zeta \maps U\Phi \To U\Phi\Phi$ is precisely equivalent to an $M$-algebra map $h \maps B \to \Phi(B)$. 
\end{lem}

\begin{proof}
    To say $\Phi$ is a lifting of $M\Alg(B, -) \maps M\Alg \to \Set$ means precisely that we have a representation $U\Phi \cong M\Alg(B, -)$. It follows that $U\Phi\Phi \cong M\Alg(B, -) \circ \Phi = M\Alg(B, \Phi(-))$, so that a transformation $\zeta \maps U\Phi \to U\Phi\Phi$ may be identified with a transformation 
    \[
        \zeta \maps M\Alg(B, -) \to M\Alg(B, \Phi(-)),
    \]
    but this corresponds to an algebra map $h \maps B \to \Phi(B)$, by the Yoneda lemma. 
\end{proof}

In more detail, the transformation $\zeta$ is retrieved from $h$ by the formula 
\[
    \zeta_C(f) = (B \xrightarrow{\;h\;} \Phi(B) \xrightarrow{\Phi(f)} \Phi(C)),
\]
or, in terms of hom-sets, by precomposing the $\Phi$-functoriality structure with $h$: 
\[
    \zeta_C = (M\Alg(B, C) \to M\Alg(\Phi B, \Phi C) \xrightarrow{M\Alg(h, 1)} M\Alg(B, \Phi C)). 
\]

Before stating our next result, we introduce the following construction: for a given algebra map $h \maps B \to \Phi(B)$ and any set $J$, we let $h_J \maps J \cdot B \to \Phi(J \cdot B)$ be the unique algebra map such that for every $j \in J$, and coproduct coprojection $i_j \maps B \to J \cdot B$, the map $h_J \circ i_j$ equals the composite 
\[ 
    B \xrightarrow{h} \Phi(B) \xrightarrow{\Phi(i_j)} \Phi(J \cdot B).
\]

\begin{lem}
\label{lem:M-plethory_square}
    A natural transformation $\delta \maps \Phi \To \Phi\Phi$ is precisely equivalent to an $M$-algebra map $h \maps B \to \Phi(B)$ such that the diagram 
    \[
    \begin{tikzcd}
        B 
        \arrow[r, "{[\theta]}"] 
        \arrow[d, "h", swap] 
        & 
        J \cdot B \arrow[d, "h_J"] 
        \\
        \Phi B 
        \arrow[r, "{\Phi([\theta])}", swap] 
        & 
        \Phi(J \cdot B)
    \end{tikzcd}
    \]
    commutes for every co-operation $[\theta] \maps B \to J \cdot B$ of the $M$-bialgebra $B$. 
\end{lem}

\begin{proof}
    To have such a transformation $\delta$ between $M$-algebra-valued functors is equivalent to having the underlying transformation $\zeta = U\delta$ preserve the operations of the monad. In other words, writing 
    \[
        U\Phi \cong M\Alg(B, -), \qquad U\Phi\Phi \cong M\Alg(B, \Phi -) 
    \]
    the condition is that for every $J$-ary operation $\theta$ we have a commutative diagram: 
    \[
    \begin{tikzcd}
         M\Alg(B, -)^J 
         \arrow[r, "\zeta^J"] 
         \arrow[swap, d, "\theta"] 
         & 
         M\Alg(B, \Phi -)^J 
         \arrow[d, "\theta"] 
         \\ 
         M\Alg(B, -) 
         \arrow[r, "\zeta", swap] 
         & 
         M\Alg(B, \Phi -).
    \end{tikzcd}
    \]
    These operations are induced from co-operations $[\theta] \maps B \to J \cdot B$ in the contravariant argument, i.e.\ we use the natural isomorphism $M\Alg(B, C)^J \cong M\Alg(J \cdot B, C)$ to rewrite the last diagram into the form 
    \[
    \begin{tikzcd}
        M\Alg(J \cdot B, -) 
        \arrow[r] 
        \arrow[swap, d, "{M\Alg(\lbrack\theta\rbrack, -)}"] 
        & 
        M\Alg(J \cdot B, \Phi -) 
        \arrow[d, "{M\Alg(\lbrack\theta\rbrack, \Phi -)}"] 
        \\ 
        M\Alg(B, -) 
        \arrow[r, "\zeta", swap] 
        & 
        M\Alg(B, \Phi -) 
    \end{tikzcd}
    \]
    where $[\theta]$ denotes the corresponding co-operation. Applying the Yoneda lemma as in the previous lemma, the commutativity of this last diagram holds if and only if it holds when applied to the identity element $1 \in M\Alg(J \cdot B, J \cdot B)$. In that case, the top horizontal arrow sends this identity $1$ to $h_J$, and the right vertical arrow sends this in turn to the composite $h_J \circ [\theta]$. The left vertical arrow sends $1$ to $[\theta]$, and the bottom horizontal arrow sends this in turn to the composite $\Phi([\theta]) \circ h$. Thus commutativity of the preceding diagram is equivalent to the equation 
    \[
        h_J \circ [\theta] = \Phi([\theta]) \circ h
    \]
    and this completes the proof. 
\end{proof} 

\begin{lem}
\label{lem:M-plethory_counit}
    A morphism $\varepsilon \maps \Phi \To 1_{M\Alg}$ precisely corresponds to a bialgebra map $\eta \maps F(1) \to B$. 
\end{lem}

\begin{proof}
    This follows from \cref{thm:MBialg}: such a morphism $\varepsilon \maps \Phi \to 1_{M\Alg}$ is equivalent to a bialgebra morphism between their representing bialgebras, and the representing bialgebra of the identity $1_{M\Alg}$ is given by the free algebra on one generator $F(1)$, with its canonical bialgebra structure.
\end{proof}

\begin{lem}
\label{lem:M-plethory_coassociativity}
    The coassociativity and counit equations for the data 
    \[
    \delta \maps \Phi \To \Phi\Phi, \qquad \varepsilon \maps \Phi \To 1_{M\Alg}
    \]
    precisely correspond to associativity and unit equations for the plethysm and plethysm unit $1 \xrightarrow{u1} UF(1) \xrightarrow{U\eta} UB$, where $u$ denotes the unit of the adjunction $F \dashv U$. 
\end{lem}

\begin{proof}
    We focus on coassociativity, leaving the counit equation to the reader. By faithfulness of $U$, the coassociativity square for $\delta$ commutes if and only if we have commutativity of the following equivalent diagrams: 
    \[
    \begin{tikzcd} 
        U\Phi 
        \arrow[r, "U\delta"] 
        \arrow[d, "U\delta", swap] 
        & 
        U\Phi\Phi 
        \arrow[d, "U\delta\Phi"] 
        \\
        U\Phi\Phi 
        \arrow[r, "U\Phi\delta", swap] 
        & 
        U\Phi\Phi\Phi
    \end{tikzcd} 
    \qquad \Leftrightarrow \qquad 
    \begin{tikzcd}
        M\Alg(B, -) 
        \arrow[r, "\zeta"] 
        \arrow[dd, "\zeta", swap]
        & 
        M\Alg(B, \Phi -) 
        \arrow[d] 
        \\& 
        M\Alg(\Phi B, \Phi\Phi -) 
        \arrow[d, "{M\Alg(h, \Phi\Phi -)}"] 
        \\
        M\Alg(B, \Phi -) 
        \arrow[swap, r, "{M\Alg(B, \delta)}"] 
        & 
        M\Alg(B, \Phi\Phi -)
    \end{tikzcd}
    \]
    where the vertical composite uses the hom-set expression of $U\delta \Phi = \zeta \Phi$ given after the proof of \cref{lem:M-plethory_h}. Chasing through the Yoneda lemma, the last diagram commutes iff the following square commutes (whose import is that $h$ defines a $\Phi$-coalgebra structure): 
    \[
    \begin{tikzcd}
        B 
        \arrow[r, "h"] 
        \arrow[d, "h", swap] 
        & 
        \Phi B 
        \arrow[d, "\Phi h"] 
        \\
        \Phi B 
        \arrow[r, "\delta B", swap] 
        & 
        \Phi\Phi B.
    \end{tikzcd}
    \] 
    Again by faithfulness of $U$, this last square commutes if and only if we have commutativity of the following diagrams: 
    \[
    \begin{tikzcd}
        UB 
        \arrow[r, "Uh"] 
        \arrow[d, "Uh", swap] 
        & 
        U\Phi B 
        \arrow[d, "U\Phi h"] 
        \\ 
        U\Phi B 
        \arrow[r, "U\delta B", swap] 
        & 
        U\Phi\Phi B
    \end{tikzcd}
    \quad
    \Leftrightarrow
    \begin{tikzcd}
        UB 
        \arrow[rr, "Uh"] 
        \arrow[d, "Uh", swap] 
        && 
        M\Alg(B, B) 
        \arrow[d, "{M\Alg(B, h)}"] 
        \\ 
        M\Alg(B, B) 
        \arrow[r] 
        &
        M\Alg(\Phi B, \Phi B) 
        \arrow[r, "{M\Alg(h, 1)}", swap] 
        &
        M\Alg(B, \Phi B)
    \end{tikzcd}
    \]
    where the horizontal composite uses the hom-set expression of $U\delta B = \zeta B$ given after the proof of \cref{lem:M-plethory_h}. 
    
    At this point, it is simplest to think of the last diagram purely in terms of sets and functions; for example, $\Phi B$ is just a hom-set equipped with $M$-algebra structure. The function $Uh$ gives a function 
    \[
        UB \to M\Alg(B, B) \to \Set(UB, UB)
    \]
    which de-curries to a map called the \define{plethystic multiplication}, denoted for now as $p \maps UB \times UB \to UB$. Chasing an element $a \in UB$ around the last diagram, the commutativity of the last diagram is equivalent to commutativity of an element assignment where going across (A) and then down (D) we arrive at 
    \[
        a \stackrel{\mathrm{A}}{\mapsto} [p(a, -) = [b \mapsto p(a, b)]] \stackrel{\mathrm{D}}{\mapsto} (b \mapsto p(p(a, b), -))
    \]
    whereas going down and then across, we arrive at 
    \[
        a \stackrel{\mathrm{D}}{\mapsto} p(a, -) \stackrel{\mathrm{A}}{\mapsto} [f \mapsto p(a, -) \circ f] \stackrel{\mathrm{A}}{\mapsto} [b \mapsto p(a, -) \circ p(b, -) = p(a, p(b, -)]
    \]
    so that the commutativity asserts precisely $p(p(a, b), -) = p(a, p(b, -))$, which is associativity for the plethysm. 
\end{proof}

\subsection{2-Plethories}
\label{sec:2-plethories}

Next we categorify the concept of rig-plethory. We believe it is possible to categorify the entire preceding story. For example, \cref{thm:liftadjoint} should categorify to the statement that any lift $\Phi \maps \TRig \to \TRig$ of a representable \break $\TRig(B, -) \maps \TRig \to \CCat$ through the forgetful functor $U \maps \TRig \to \CCat$ must be a right biadjoint. However, proving this would require a detour through some 2-categorical algebra, including for example a 2-categorical analogue of the adjoint lifting theorem \cite{Lucatelli}. Since we have a number of equivalent definitions of $M$-plethory, we prefer to categorify the one most convenient for our purposes. So, we define a 2-plethory as a right 2-adjoint 2-comonad:

\begin{defn}
    A \define{2-plethory} is a 2-comonad $\Phi \maps \TRig \to \TRig$ whose underlying 2-functor is a right 2-adjoint. 
\end{defn}

The underlying 2-rig of $\Phi$ may be extracted by an easy categorification of \cref{lem:rightrepable}: 
\begin{prop}
\label{prop:2rep}    
    Let $U \maps \TRig \to \CCat$ be the forgetful functor, and $\Phi$ a 2-plethory. Then $U\Phi \maps \TRig \to \CCat$ is representable, and the representing object is $\Psi(\ksbar)$, where $\Psi$ is left 2-adjoint to $\Phi$.  
\end{prop}

The representing object $B = \Psi(\ksbar)$ is the underlying 2-rig of the 2-plethory $\Phi$; the fact that $U\Phi \simeq \TRig(B, -)$ means that $\Phi$ is a lift of $\TRig(B, -)$ through $U$, so that $\Phi$ endows $B$ with 2-birig structure as in this by now all too familiar picture:
\[
\begin{tikzcd}[column sep = large]
    &
    \TRig
    \arrow[d, "U"]
    \\
    \TRig
    \arrow[ur, "\Phi"]
    \arrow[ur, swap, ""{name = phi}, phantom, bend right = 5, pos = 0.6]
    \arrow[r, swap, "{\TRig(B, -)}"]
    \arrow[from = phi, r, Rightarrow, "\sim", sloped]
    &
    \CCat.
\end{tikzcd}
\]

\begin{proof}
    By the proof of \cref{thm:representing_U}, $U$ has a left 2-adjoint taking any category $C$ to the 2-rig $\overline{k\SS(C)}$. Thus $U\Phi \maps \TRig \to \CCat$ is a right 2-adjoint, with left 2-adjoint $L$ say, and we have 
    \[
    U\Phi \simeq \C(1, U\Phi-) \simeq \TRig(L(1), -).
    \]
    making $L(1)$ a representing object for $U\Phi$. But $L$, being the left 2-adjoint of $U\Psi$, is equivalent to $\Psi \circ \overline{k\SS(-)}$ where $\Psi$ is the left 2-adjoint of $\Phi$. Thus $L(1) \simeq \Psi(\overline{k\SS(1)}) = \Psi(\ksbar)$.
\end{proof}

\begin{thm}
\label{thm:ksbar_as_2-plethory}
    The identity $1 \maps \TRig \to \TRig$ is a 2-plethory with underlying 2-rig $\ksbar$, which in turn has underlying category $U(\ksbar) \simeq \Poly \simeq \Schur$.
\end{thm}

\begin{proof}
    The identity $1 \maps \TRig \to \TRig$ is naturally a 2-comonad whose underlying 2-functor has a left 2-adjoint, namely the identity.  The result then follows from \cref{prop:2rep} taking $\Phi = \Psi = 1$, together with Propositions \ref{prop:poly_2rig} and \ref{prop:schur_2rig}, which say that $\Poly$ and $\Schur$ are equivalent to the underlying category of $\ksbar$.
\end{proof}

Any 2-plethory $\Phi$ has a left 2-adjoint $\Psi$, and the 2-comonad structure on $\Phi$ is mated to a 2-monad structure on $\Psi$.  The representing object for $U\Phi$ is $B = \Psi(\ksbar)$, so the 2-monad multiplication $m \maps \Psi\Psi \To \Psi$ applied to $\ksbar$ results in a 2-rig map 
\[
    \Psi(B) \to B
\]
which by the adjunction $\Psi \dashv \Phi$ transforms to a 2-rig map
\[
    h \maps B \to \Phi(B)
\]
Applying $U \maps \TRig \to \CCat$, this results in a functor 
\[
    UB \to \TRig(B, B)
\]
whose codomain maps into $\CCat(UB, UB)$. From the cartesian closure of $\CCat$, the resulting functor $UB \to \CCat(UB, UB)$ exponentially transposes to a functor 
\[
    \comp_B \maps UB \times UB \to UB
\]
called the \define{plethystic monoidal structure} for the 2-plethory $\Phi$. 

The reader may wonder how this abstract description of the plethystic monoidal structure reduces, in a special case, to the substitution product described in \cref{cor:pleth}.  For this we take the 2-plethory $\Phi$ to be $1 \maps \TRig \to \TRig$, so that $B = \ksbar$.  In this case, the curried form of $\comp_B$ is the map
\[ Uh \colon U(\ksbar) \to \TRig(\ksbar, \ksbar)\]
that takes any $\tau \in U(\ksbar)$ to the unique (up to isomorphism) 2-rig map $\ksbar \to \ksbar$ whose value at the generating object of the 2-rig $\ksbar$ is $\tau$.  Explicitly, this is none other than the assignment 
\[\tau \mapsto (\rho \mapsto \rho \bullet \tau)\] 
described in \cref{cor:pleth}.  Thus, the plethystic monoidal product for $1_{\TRig}$ is indeed the substitution product of polynomial species.

\section{The rig-plethory of positive symmetric functions}
\label{sec:lambdaplus}

Our next major goal is to decategorify the 2-rig $\ksbar$ by forming its Grothendieck group, $\Lambda = K(\ksbar)$ and show that the result is a ring-plethory.  In this section we start by showing that the 2-birig structure on $\ksbar$ induces a birig structure on the set of isomorphism classes of objects of $\ksbar$, which we call $\Lambda_+$.  Then we show that the 2-plethory structure on $\ksbar$ makes $\Lambda_+$ into a rig-plethory. Some of this depends on special features of the representation theory of the symmetric groups. In the next section we  explain how the rig-plethory structure on $\Lambda_+$ induces a ring-plethory structure on $\Lambda$. 

It helps to break down the construction of the Grothendieck group $K$ into two steps. For the first step, given a Cauchy complete linear category $\C$, we define $J(\C)$ to be the set of isomorphism classes of objects of $\C$. This is made into a commutative monoid with its addition and additive unit coming from coproducts and the initial object in $\C$:
\[     [x] + [y] = [x \oplus y], \qquad 0 = [0]  \]
for all $x, y \in \C$. In fact $J$ extends to a 2-functor 
\[   J \maps \Cauch\Lin_0 \to \CMon \]
where we treat $\CMon$ as a 2-category with only identity 2-morphisms, and for any 2-category $\mathbf{B}$ we let $\mathbf{B}_0$ be the sub-2-category with the same objects and morphisms, but only invertible 2-morphisms as 2-morphisms. The point here is that naturally isomorphic functors have the same effect on isomorphism classes of objects.

Similarly, if $\R$ is a 2-rig, then $J(\R)$ acquires a rig structure with its multiplication and multiplicative unit coming from monoidal products and the monoidal unit in $\R$: 
\[
    [x] \cdot [y] = [x \otimes y], \qquad 1 = [I].
\]
In particular, we call $J(\ksbar)$ the rig of \define{positive symmetric functions}, and we denote it as $\Lambda_+$.

The second step is to take the group completion of the commutative monoid $J(\C)$. This gives the Grothendieck group $K(\C)$. The theory of rings and their modules extends to rigs, and just as an abelian group is the same as an $\Z$-module, a commutative monoid may be seen as an $\N$-module. In these terms group completion is the functor 
\[
\Z \otimes_\N - \maps \CMon \to \Ab.
\]
Thus we have a commutative diagram
\[
\begin{tikzcd}
    \Cauch\Lin_0
    \arrow[r, swap, "J"] \arrow[rr, bend left, "K"]
    &
    \CMon 
    \arrow[r, swap, "\Z \otimes_\N -"]
    &
    \Ab
\end{tikzcd}
\]
Similarly, if $\R$ is a 2-rig, then by group-completing the additive monoid of the rig $J(\R)$, we get a ring $K(\R)$, so that we have a commutative diagram
\[
\begin{tikzcd}[column sep = large]
    \TRig_0
    \arrow[r, swap, "J"] 
    \arrow[rr, bend left, "K"]
    & 
    \Rig
    \arrow[r, swap, "\Z \otimes_\N -"]
    &
    \Ring
\end{tikzcd}
\]
In particular, the ring $K(\ksbar)$ is denoted $\Lambda$. This is the famous ring of symmetric functions. 

In the rest of this paper, we transport the conceptually simple 2-plethory structure on $\ksbar$ to a ring-plethory structure on $\Lambda$ using the functor $K$. This is not trivial, and we shall have to consider the functors $J$ and $\Z \otimes_\N -$ separately in what follows. 

Some technical considerations include the following. We would like for the 2-birig structure on $\overline{k \S}$, which involves 2-rig co-operations of type 
\[
    \overline{k \S} \to \overline{k \S} \boxtimes \overline{k \S}, 
\]
to yield rig co-operations of type
\[
    J(\overline{k \S}) \to J(\overline{k \S}) \otimes J(\overline{k \S})
\]
making $J(\overline{k \S})$ into a birig. This would hold automatically if $J \maps (\Cauch\Lin_0, \boxtimes) \to (\CMon, \otimes)$ preserved tensor products. As we shall see, this is not true in complete generality. But it turns out to be true for the tensor products we are interested in, involving $\overline{k \S}$. 

A separate issue is that we need to examine when the group completion of a birig is a biring, and when the group completion of a rig-plethory gives a ring-plethory. We reserve those considerations for \cref{sec:lambda}.

\subsection{The rig structure on $\Lambda_+$}

We begin by making $\Lambda_+$ into a rig. To do this, we lift the 2-functor
\[
    J \maps \Cauch\Lin_0 \to \CMon
\]
to a 2-functor from 2-rigs to rigs. Applying this to the 2-rig $\ksbar$, we obtain the rig structure on $\Lambda_+$.

As a warm-up, we begin with a more detailed consideration of $J$ from the viewpoint of base change for enriched categories. To bring $\Cauch\Lin$ and $\CMon$ together on a level playing ground, we performed two moves. The first was the move from $\Cauch\Lin$ to $\Cauch\Lin_0$ by discarding all 2-cells except for the invertible ones. The second was to treat $\CMon$ as a 2-category whose only 2-cells are identity 2-cells. 

Both moves involve change of base. If we consider a 2-category (like $\Cauch\Lin$) as a $\Cat$-enriched category, i.e.\ with homs valued in $\Cat$,
\[
\hom \maps  \ob(\Cauch\Lin) \times \ob(\Cauch\Lin) \to \Cat,
\]
then the first move amounts to composing $\hom$ with the functor 
\[
    \core \maps \Cat \to \Grpd
\]
that assigns to a category $\C$ the groupoid $\core(\C)$ whose objects are the same as those of $\C$, and whose arrows are the invertible arrows of $\C$. (Note that this functor $\core$ is a 1-functor only, not a 2-functor. Note also that this 1-functor preserves cartesian products. In general, for change of base of enrichment to work properly, one needs at a minimum a lax monoidal functor from one base to the other; for cartesian monoidal products, this is the same as (strong) preservation of cartesian products.) In general, for any 2-category $\bC$, the composite
\[
    \ob(\bC) \times \ob(\bC) \xrightarrow{\hom} \Cat \xrightarrow{\core} \Grpd
\]
defines a $\Grpd$-enriched category $\bC_0$. 

On the other hand, if we consider an ordinary 1-category (like $\CMon$) as a $\Set$-enriched category, i.e.\ with homs valued in $\Set$, 
\[
    \hom \maps \ob(\CMon) \times \ob(\CMon) \to \Set, 
\]
then the second move amounts to composing this $\hom$ with the product-preserving functor from sets to groupoids, 
\[
    \disc \maps \Set \to \Grpd,
\]
that assigns to a set $X$ the discrete groupoid $\disc(X)$ whose objects are the elements of $X$. In general, for any category $\C$, the composite
\[
    \ob(\C) \times \ob(\C) \xrightarrow{\hom} \Set \xrightarrow{\disc} \Grpd
\]
defines a $\Grpd$-enriched category.

By these moves, the homs are brought to a level playing field in $\Grpd$, and $J$ is construed as a $\Grpd$-enriched functor, involving maps of groupoids
\[
    \core(\hom(\C, \D)) \to \disc(\hom(J\C, J\D)). 
\]
Now, $\disc \maps \Set \to \Grpd$ has a left adjoint $\pi_0 \maps \Grpd \to \Set$ which assigns to a groupoid its set of `connected components', or isomorphism classes. Thus the above maps of groupoids are equivalent to functions between sets 
\[
    (\pi_0 \circ \core)(\hom(\C, \D)) \to \hom(J\C, J\D)
\]
and since it is easily seen that $\pi_0$ preserves products, we obtain a product-preserving change of base, 
\[
    \pi_0\circ\core \maps \Cat \to \Set. 
\]

In many cases, applying the base change  $\pi_0 \circ \core$ is the `right' way to turn a 2-category into a 1-category. It is, in fact, simply a formalization of `decategorification', which demotes isomorphisms to equalities. The construction is sometimes called the \define{homotopy category} of a 2-category, so we write $\bC\ho$ for this construction applied to a 2-category $\bC$. 

We now apply this to our goal: decategorifying a 2-rig and obtaining a rig. First, note that $\TRig\ho$ is the 1-category whose objects are 2-rigs and whose morphisms are symmetric monoidal natural isomorphism classes of 2-rig maps. We define $H \maps \TRig\ho \to \Set$ to be the functor given by the composite
\[
    \TRig\ho \xrightarrow{U\ho} \CCat\ho \xrightarrow{\hom(1, -)} \Set. 
\]
This functor $H$ assigns to a small 2-rig its set of isomorphism classes. 

Another point of view on this functor will be useful to us. Observe that the 1-functor between 1-categories 
\[
    \core \maps \Cat \to \Grpd
\]
extends to a \emph{2-functor} 
\[
   \mathbf{core} \maps \CCat_0 \to \GGrpd
\]
from the $\Grpd$-enriched category $\CCat_0$ consisting of categories, functors, and natural isomorphisms to the \emph{2-category} of groupoids, functors and natural isomorphisms. We also have an evident 2-functor $\pi_0 \maps \GGrpd \to \Set$. Let
\[
    \pi \maps \CCat_0 \to \Set
\]
denote the composite 2-functor $\pi_0 \circ \mathbf{core}$.

Finally, to $\pi$ we may apply the homotopy 1-category construction, obtaining a functor 
\[
    \CCat\ho \simeq (\CCat_0)\ho  \xrightarrow{\pi\ho} \Set\ho \simeq \Set. 
\]
\begin{lem}
    The above functor from $\CCat\ho$ to $\Set$ is naturally isomorphic to
    \[ \hom(1, -) \maps \CCat\ho \to \Set.\]
\end{lem}
\begin{proof}
    For any category $C \in \CCat\ho$, $\hom(1, C)$ is the set of isomorphism classes of objects $[c]\maps 1 \to \C$ of $\C$, and for any class of functors $[F] \maps \C \to \D$, $\hom(1, [F])$ is the (well-defined) function taking $[c]$ to $[F(c)]$. If $\alpha \maps F \To G$ is a 2-cell of $\CCat_0$, then $\pi(\alpha)$ is the identity 2-cell between functions $\pi(F), \pi(G)$ where $\pi(F)(c) = [F(c)]$. So $\pi\ho$ has the same effect as $\hom(1, -)$ on 0-cells and 1-cells of $\C\ho$. 
\end{proof}

\begin{cor}
\label{cor:hom-change}
    Let $\C$ be a 2-category, and for objects $c, d$ of $\C$, regard the hom-category $\C(c, d)$ as an object of $\CCat\ho$. Then 
    \[
        \CCat\ho(1, \C(c, d)) \cong \C\ho(c, d).
    \]
\end{cor}

With this infrastructure in place, we can begin the decategorification process. Starting with any 2-rig $R$, the underlying category $U(R)$ comes equipped with categorical operations making the category a 2-rig, e.g.\ the coproduct operation and tensor operation 
\[
    \oplus, \otimes \maps UR \times UR \to UR
\]
and so on. There are various coherent natural isomorphisms such as associativity, symmetry, distributivity, etc.\ but when we change our base of enrichment from $\Cat$ to $\Set$ along $\pi_0 \circ \core$, and interpret these isomorphisms as equations in $\CCat\ho$, the object $UR$ becomes simply a rig object in $\CCat\ho$. Then we apply the product-preserving functor $\CCat\ho(1, -) \maps \CCat\ho \to \Set$ to this rig object to get a set $H(R)$ equipped with rig structure. As above, this rig is denoted $J(R)$. 

Similarly, the underlying functor of a 2-rig map $R \to S$ preserves these categorical operations up to coherent isomorphisms, but these isomorphisms become equations when in $\CCat\ho$, so that we get a rig homomorphism $U(R) \to U(S)$ between rig objects in $\CCat\ho$. Applying $\CCat\ho(1, -) \maps \CCat\ho \to \Set$, the function $H(R) \to H(S)$ preserves the rig operations, giving a homomorphism $J(R) \to J(S)$ in $\Rig$. 

Thus we have the following result. 

\begin{lem}
    The functor $H \maps \TRig\ho \to \Set$ lifts through the forgetful functor $U \maps \Rig \to \Set$ to a product-preserving functor $J \maps \TRig\ho \to \Rig$: 
    \[
    \begin{tikzcd}[column sep = large]
        &
        \Rig
        \arrow[d, "U"]
        \\
        \TRig\ho
        \arrow[r, swap, "H"]
        \arrow[ur, dashed, "J"]
        &
        \Set.
    \end{tikzcd}
    \]
\end{lem}

\begin{proof} 
    The product-preservation of $J$ follows from product-preservation of $H$ and the fact that $U$ reflects products (products in $\Rig$ are created from products in $\Set$). 
\end{proof}

Applying this to the 2-rig $\ksbar$ we obtain:

\begin{thm}
$\Lambda_+ = J(\ksbar)$ is a rig, so its group completion $\Lambda = \mathbb{Z} \otimes_\N \Lambda_+$ is a ring.
\end{thm}

\subsection{The birig structure on $\Lambda_+$}

Next we decategorify the co-operations on $\ksbar$ to make $\Lambda_+$ into a birig. The composite functor 
\[
\Cauch\Lin\ho \xrightarrow{U\ho} \CCat\ho \xrightarrow{\CCat\ho(1, -)} \Set
\]
lifts through the forgetful functor $U \maps \CMon \to \Set$ to a functor 
\[
    J \maps \Cauch\Lin\ho \to \CMon
\]
since coproducts in Cauchy complete linear categories induce addition on  isomorphism classes of objects. We would be all set if this $J$ were a strong monoidal functor, and thus equipped with natural isomorphisms
\[
    J(\C) \otimes_\N J(\D) \cong J(\C \boxtimes \D)
\]
for all Cauchy complete linear categories $\C, \D$. The co-operations on $\ksbar$ would then give co-operations on $J(\ksbar) = \Lambda_+$ making $\Lambda_+$ into a birig.

Alas, this fails in general. We certainly have a natural map $J(\C) \otimes_\N J(\D) \to J(\C \boxtimes \D)$, making $J$ into a \emph{lax} monoidal functor. The trouble is that this map is not always onto. For example, take $k = \mathbb{R}$ and let $\C$ be the linear category of finite-dimensional real representations of $\mathbb{C}$, regarded as an algebra over $\mathbb{R}$. This linear category is Cauchy complete, and every object in it is a finite coproduct of copies of $\mathbb{C}$. We thus have $J(\C) \cong \N$, and it follows that $J(\C) \otimes_\N J(\C) \cong \N$. On the other hand $\J(\C \boxtimes \C)$ is equivalent to the category of representations of $\mathbb{C} \otimes_\R \mathbb{C} \cong \mathbb{C} \oplus \mathbb{C}$, so by a similar argument $\J(\C \boxtimes \C) \cong \N^2$. 

This sort of problem can occur whenever $k$ is not algebraically closed. However, in the case of $\ksbar$ we are more fortunate.   The 2-rig $\ksbar^{\boxtimes n}$, considered in the homotopy 1-category $\TRig\ho$, is the coproduct of $n$ copies of $\ksbar$, so we have $n$ coproduct coprojections 
\[
i_1, \ldots, i_n\maps  \ksbar \to \ksbar^{\boxtimes n}.
\]
Since $J(\ksbar)^{\otimes n}$ is the coproduct in $\Rig$ of $n$ copies of $J(\ksbar)$, the maps $J(i_1), \ldots, J(i_n)$ collate into a single rig map 
\[
J(\ksbar)^{\otimes n} \to J(\ksbar^{\boxtimes n})
\]
and in fact this is an isomorphism.

\begin{lem}
\label{lem:J-preserves-coproducts}
The canonical rig map $J(\ksbar)^{\otimes n} \to J(\ksbar^{\boxtimes n})$ is an isomorphism.  
\end{lem}

\begin{proof}  
First, note that if $\C, \D \in \Cauch\Lin\ho$ and $\C$ is a coproduct of copies of $\Fin\Vect$, the canonical map of commutative monoids
\[  J(\C) \otimes_\N J(\D) \to J(\C \boxtimes \D) \]
is an isomorphism.  To see this,  assume $\C \simeq X \cdot \Fin\Vect$ for some set $X$, where we write $X \cdot$ to mean $X$-fold coproduct.  Then $J(\C) \cong X \cdot \N$ as commutative monoids, so
\[ J(\C) \otimes_\N J(\D) \simeq (\X \cdot \N) \otimes J(\D) \simeq X \cdot J(\D) \simeq J(X \cdot \D) \simeq J((X \cdot \Fin\Vect) \boxtimes \D) \simeq J(\C \boxtimes \D).\]

Since the canonical rig map $J(\ksbar)^{\otimes n} \to J(\ksbar^{\boxtimes n})$ can be obtained inductively from the canonical maps $J(\C) \otimes_\N J(\D) \to J(\C \boxtimes \D)$, it suffices to show that $\ksbar$ is, as an object of $\Cauch\Lin\ho$, a coproduct of copies of $\Fin\Vect$.

A field $k$ is a `splitting field' for a finite group $G$ precisely when its category $\Rep(G)$ of finite-dimensional representations is equivalent, as linear category, to a coproduct of copies of $\Fin\Vect$.   In this case $\Rep(G)$ is also a coproduct of copies of $\Fin\Vect$ when regarded as an object of $\Cauch\Lin\ho$.  It is well-known that that any field of characteristic zero is a splitting field for the symmetric group $S_n$ \cite[Corollary 4.16]{Lorenz}.  Since in $\Cauch\Lin\ho$ we have
\[   \ksbar \cong \sum_{m = 0}^\infty \Rep(S_n) \]
it follows that for any field $k$ of characteristic zero, $\ksbar$ is a coproduct of copies of $\Fin\Vect$ in $\Cauch\Lin\ho$.  
\end{proof}

This result immediately leads to a birig structure on $\Lambda_+ = J(\ksbar)$: each co-operation for the 2-birig $\ksbar$ induces a corresponding co-operation on $J(\ksbar)$. For example, the coaddition $\alpha \maps \ksbar \to \ksbar \boxtimes \ksbar$ pertaining to $\ksbar$ as a 2-corig object in $\TRig$ gives the coaddition $\alpha$ for $\ksbar$ as a corig object in $\TRig\ho$. Applying $J \maps \TRig\ho \to \Rig$, we define the coaddition on $J(\ksbar)$ to be the composite in $\Rig$
\[
J(\ksbar) \stackrel{J(\alpha)}{\to} J(\ksbar \boxtimes \ksbar) \cong J(\ksbar) \otimes J(\ksbar)
\]
by virtue of \cref{lem:J-preserves-coproducts}. The remaining co-operations are defined by a similar procedure, and the following result is clear. 

\begin{thm}
\label{thm:lambdaplus-forms-birig}
   $\Lambda_+ = J(\ksbar)$ with the co-operations thus defined is a birig. 
\end{thm}

\subsection{The rig-plethory structure on $\Lambda_+$}

Our next task is to establish a rig-plethory structure on $\Lambda_+$. The birig structure on $\Lambda_+$ induces a rig structure on the representable functor $\Rig(\Lambda_+, -)$, giving in particular a rig structure on the hom-set $\Rig(\Lambda_+, \Lambda_+)$, and a subsidiary task is to construct a suitable rig map
\[
h \maps \Lambda_+ \to \Rig(\Lambda_+, \Lambda_+)
\]
that represents plethysm. 

The function $h \maps U_\Rig\Lambda_+ \to \Rig(\Lambda_+, \Lambda_+)$ is constructed from the 2-plethysm on $\ksbar$, given by a functor
\[
U(\ksbar) \to \TRig(\ksbar, \ksbar)
\]
living as a morphism in $\CCat$. We treat this as a morphism in $\CCat\ho$, and then apply $\CCat\ho(1, -) \maps \CCat\ho \to \Set$. The result is a function 
\[
H(\ksbar) \to \TRig\ho(\ksbar, \ksbar)
\]
with the help of \cref{cor:hom-change}. Then compose the above function with the map 
\[
\TRig\ho(\ksbar, \ksbar) \to \Rig(J(\ksbar), J(\ksbar)) = \Rig(\Lambda_+, \Lambda_+)
\]
arising from the functoriality of $J$. The resulting composite is the desired function 
\[
h \maps U_\Rig(\Lambda_+) = U_\Rig(J(\ksbar)) = H(\ksbar) \to \Rig(\Lambda_+, \Lambda_+).
\]

The next battery of results show that $h$ gives $\Lambda_+$ a rig-plethory structure, following the general breakdown of $M$-plethories given in \cref{lem:M-plethory_h}, \cref{lem:M-plethory_square}, \cref{lem:M-plethory_counit}, and \cref{lem:M-plethory_coassociativity}. It should be noted the proof techniques follow essentially the same pattern over and over, used throughout the remainder of this section: 

\begin{itemize}
    \item Observe the analogous condition at the 2-rig level. This is usually trivial because the 2-plethory itself is trivial, being just an identity 2-comonad on 2-rigs. 
    \item By changing the base of enrichment from $\Cat$ to $\Set$, observe the same condition as a strictly commuting diagram in $\TRig\ho$. 
    \item Apply $H \maps \TRig\ho \to \Set$. Usually $H$ lifts through $J \maps \TRig\ho \to \Rig$, and moreover it converts hom-categories or hom-2-rigs of type $\TRig(R, S)$ appearing in the last step into hom-sets or hom-rigs $\TRig\ho(R, S)$. 
    \item Postcompose with maps $\TRig\ho(R, S) \to \Rig(J(R), J(S))$ that express functoriality of $J$, and combine with properties of $J$ (preservation of products, preservation of suitable copowers) to establish the corresponding condition at the rig level. 
\end{itemize}

We begin by checking the condition in \cref{lem:M-plethory_h}.

\begin{lem}
\label{lem:h-is-rig-map}
   The map $h$ is a rig homomorphism. 
\end{lem}

\begin{proof}
    We indicate why $h$ preserves the rig multiplication; preservation of the other operations is treated similarly. Letting $\mu$ denote comultiplication, we have a 2-rig map $U(\ksbar) \to \TRig(\ksbar, \ksbar)$ which means we have a 2-cell isomorphism 
    \[
    \begin{tikzcd}
        U(\ksbar) \times U(\ksbar)
        \arrow[r]
        \arrow[dd, swap, "m"]
        &
        \TRig(\ksbar, \ksbar) \times \TRig(\ksbar, \ksbar)
        \arrow[d, sloped, "\sim"]
        \\&
        \TRig(\ksbar \boxtimes \ksbar, \ksbar)
        \arrow[d, "{\TRig(\mu, 1)}"]
        \\
        U(\ksbar)
        \arrow[uur, phantom, "\cong"]
        \arrow[r]
        & 
        \TRig(\ksbar, \ksbar)
    \end{tikzcd}
    \]
    in $\CCat$. This 2-cell becomes an equation in $\CCat\ho$. Applying $
    \CCat\ho(1, -) \maps \CCat\ho \to \Set$ 
    together with \cref{cor:hom-change}, we obtain the following commutative diagram in $\Set$:
    \[
    \begin{tikzcd}
        H(\ksbar) \times H(\ksbar)
        \arrow[r]
        \arrow[dd, swap, "m"]
        &
        \TRig\ho(\ksbar, \ksbar) \times \TRig\ho(\ksbar, \ksbar)
        \arrow[d, sloped, "\sim"]
        \\&
        \TRig\ho(\ksbar \boxtimes \ksbar, \ksbar)
        \arrow[d, "{\TRig\ho(\mu, 1)}"]
        \\
        H(\ksbar)
        \arrow[r]
        & 
        \TRig\ho(\ksbar, \ksbar).
    \end{tikzcd}
    \]
    Abbreviating $Q \times Q$ to $Q^2$ to conserve space, we append to this last diagram another diagram which instantiates functoriality of $J$, and we invoke \cref{lem:J-preserves-coproducts}: 
    \[
    \begin{tikzcd}[column sep = tiny]
        H(\ksbar)^2
        \arrow[r]
        \arrow[ddd, swap, "m"]
        &
        \TRig\ho(\ksbar, \ksbar)^2
        \arrow[dd, sloped, "\sim"]
        \arrow[r]
        &
        \Rig(J(\ksbar), J(\ksbar))^2
        \arrow[d]
        \\&
        &
        \Rig(J(\ksbar)\otimes J(\ksbar), J(\ksbar))
        \arrow[d, sloped, "\sim"]
        \\
        & 
        \TRig\ho(\ksbar \boxtimes \ksbar, \ksbar)
        \arrow[d, "{\TRig\ho(\mu, 1)}"]
        \arrow[r]
        &
        \Rig(J(\ksbar \boxtimes \ksbar), J(\ksbar))
        \arrow[d, "{\Rig(J(\mu), 1)}"]
        \\
        H(\ksbar)
        \arrow[r]
        &
        \TRig\ho(\ksbar, \ksbar)
        \arrow[r]
        &
        \Rig(J(\ksbar), J(\ksbar))
    \end{tikzcd}
    \]
    The perimeter of this last diagram shows that $h$ preserves rig multiplication. 
\end{proof}

Next we check the condition in \cref{lem:M-plethory_square}.

\begin{lem}
\label{lem:condition-one}
    The map $h$ satisfies the commutativity condition of \cref{lem:M-plethory_square}.
\end{lem}

\begin{proof}
    The pseudonatural equivalence 
    \[
    \eta\maps 1_{\TRig} \To \Phi_{\ksbar}
    \]
    on $\TRig$ has as its component at $\ksbar$ the canonical 2-rig map 
    \[
    \eta' = \eta(\ksbar) \maps \ksbar \to \Phi_{\ksbar}(\ksbar) 
    \]
    that provides the 2-plethysm on $\ksbar$. The map 
    \[
    \eta_n' = \eta(\ksbar^{\boxtimes n}) \maps \ksbar^{\boxtimes n} \to \Phi_{\ksbar}(\ksbar^{\boxtimes n})
    \]
    is, by pseudonaturality of $\eta$, the (unique up to isomorphism) 2-rig map that makes the following square commute up to isomorphism:
    \[
    \begin{tikzcd}[column sep = large]
        \ksbar 
        \arrow[r, "\eta' = \eta(\ksbar)"] 
        \arrow[d, swap, "i_j"] 
        & 
        \Phi_{\ksbar}(\ksbar) 
        \arrow[d, "\Phi_{\ksbar}(i_j)"] 
        \\ 
        \ksbar^{\boxtimes n} 
        \arrow[r, swap, "\eta_n' = \eta(\ksbar^{\boxtimes n})"] 
        & 
        \Phi_{\ksbar}(\ksbar^{\boxtimes n}).
    \end{tikzcd}
    \]
    It follows, again by pseudonaturality, for each 2-birig co-operation $\lbrack\Theta\rbrack \maps \ksbar \to \ksbar^{\boxtimes n}$ we have a square in $\TRig$ that commutes up to 2-cell isomorphism:
    \[
    \begin{tikzcd}
        \ksbar \arrow[r, "\eta'"] \arrow[d, swap, "\lbrack\Theta\rbrack"] & \Phi_{\ksbar}(\ksbar) \arrow[d, "\Phi_{\ksbar}(\lbrack\Theta \rbrack)"] \\ \ksbar^{\boxtimes n} \arrow[r, swap, "\eta_n'"] & \Phi_{\ksbar}(\ksbar^{\boxtimes n})
    \end{tikzcd}
    \]
    This is the 2-rig form of the commutative square in \cref{lem:M-plethory_square}. 
    
    Changing our base of enrichment along $\pi_0 \circ \mathsf{core} \maps \Cat \to \Set$ gives a strictly commuting square in $\TRig\ho$. We apply $J$ to this square to get a commuting square in $\Rig$. Now, for any 2-rig $R$, the underlying set of $J(\Phi_{\ksbar}(R))$ is 
    \[
    \CCat\ho(1, U\ho\Phi_{\ksbar}(R)) \cong \CCat\ho(1, \TRig(\ksbar, R)) \cong  \TRig\ho(\ksbar, R)
    \]
    This set acquires a rig structure, using the corig structure that $\ksbar$ has as an object in $\TRig\ho$. Thus the diagram of rigs obtained by applying $J$ may be written in the form 
    \[
    \begin{tikzcd}
        & 
        J(\ksbar) 
        \arrow[r, "J(\eta')"] 
        \arrow[dl] 
        \arrow[d, "J(\lbrack\Theta \rbrack)"] 
        & 
        \TRig\ho(\ksbar, \ksbar) 
        \arrow[d, "{\TRig\ho(\ksbar, \lbrack \Theta\rbrack)}"] 
        \\ 
        J(\ksbar)^{\otimes n} 
        \arrow[r, "\sim", swap] 
        & 
        J(\ksbar^{\boxtimes n}) 
        \arrow[r, "J(\eta_n')", swap] 
        & 
        \TRig\ho(\ksbar, \ksbar^{\boxtimes n})
    \end{tikzcd}
    \]
    and we append to this diagram a square that expresses functoriality of $J$: 
    \[
    \begin{tikzcd}
        J(\ksbar)
        \arrow[r]
        \arrow[swap, d, "\lbrack \theta \rbrack = J(\lbrack\Theta\rbrack)"]
        &
        \TRig\ho(\ksbar, \ksbar)
        \arrow[d, "{\TRig\ho(1, \lbrack\Theta\rbrack)}"]
        \arrow[r]
        &
        \Rig(J(\ksbar), J(\ksbar))
        \arrow[d, "{\Rig(1, \lbrack\theta \rbrack)}"]
        \\ 
        J(\ksbar)^{\otimes n} 
        \arrow[swap, drr, "h_n"] 
        \arrow[r, "\sim"] 
        & 
        \TRig\ho(\ksbar, \ksbar^{\boxtimes n}) 
        \arrow[r] 
        & 
        \Rig(J(\ksbar), J(\ksbar^{\boxtimes n})) 
        \arrow[d, "\sim", sloped] 
        \\&& 
        \Rig(J(\ksbar), J(\ksbar)^{\otimes n}).
    \end{tikzcd}
    \]
    The top horizontal composite is the map $h$ of \cref{lem:h-is-rig-map}. The bottom arrow is $h_n$, using the fact that $J$ preserves coproducts and a simple diagram chase. Hence this last diagram manifests the square whose commutativity is:
    \[
    \begin{tikzcd}
        \Lambda_+ 
        \arrow[r, "h"] 
        \arrow[swap, d, "\lbrack\theta \rbrack"] 
        & 
        \Phi_{\Lambda_+}(\Lambda_+) 
        \arrow[d, "\Phi_{\Lambda_+}(\lbrack\theta \rbrack)"] 
        \\
        \Lambda_+^{\otimes n} 
        \arrow[r, "h_n", swap] 
        & 
        \Phi_{\Lambda_+}(\Lambda_+^{\otimes n}) 
    \end{tikzcd}  \qedhere
    \]
\end{proof}

Next, the unit of our rig-plethory is the unique rig map 
\[
I \maps \N[x] \to J(\ksbar) = \Lambda_+
\]
that takes $x$ to $[k \cdot \S(-, 1)]$, where $k \cdot \S(-, 1)$ is the plethysm unit of the 2-plethory $\ksbar$; equivalently, the canonical generator of $\ksbar$ as free 2-rig. We now check that this rig map satisfies the condition in \cref{lem:M-plethory_counit}.

\begin{lem}
\label{prop:I-birig-map}
    The rig homomorphism $I \maps \N[x] \to \Lambda_+$ is a birig homomorphism. 
\end{lem} 

\begin{proof}
    The proof is a straightforward computation; we check that $I$ preserves coaddition to illustrate. The 2-rig coaddition $\alpha \maps \ksbar \to \ksbar \boxtimes \ksbar$ is the unique 2-rig map that takes the generator $X = k \cdot \S(-, 1)$ to $X \boxtimes 1 \oplus 1 \boxtimes X$. The composite 
    \[
    J(\ksbar) \stackrel{J(\alpha)}{\to} J(\ksbar \boxtimes \ksbar) \cong J(\ksbar) \otimes J(\ksbar)
    \]
    which defines coaddition on $\Lambda_+ = J(\ksbar)$ takes the isomorphism class $[X]$ to $[X]\otimes [1] + [1] \otimes [X]$. Hence $\alpha_{\Lambda_+} \circ I$ takes $x$ to $[X]\otimes [1] + [1] \otimes [X]$. Clearly 
    \[
    \N[x] \to \N[x] \otimes \N[x] \xrightarrow{I \otimes I} \Lambda_+ \otimes \Lambda_+ 
    \]
    takes $x$ also to $[X]\otimes [1] + [1] \otimes [X]$, since the coaddition on $\N[x]$ takes $x$ to $x \otimes 1 + 1 \otimes x$. Since the two legs of 
    \[
    \begin{tikzcd}
        \N[x] 
        \arrow[r, "I"] 
        \arrow[swap, d, "\alpha"] 
        & 
        \Lambda_+ 
        \arrow[d, "\alpha_{\Lambda_+}"] 
        \\
        \N[x] \otimes \N[x] 
        \arrow[r, "I \otimes I", swap] 
        & 
        \Lambda_+ \otimes \Lambda_+
    \end{tikzcd}
    \]
    take $x$ to the same element, they must be the same rig map. 
\end{proof}

Finally we check the condition in \cref{lem:M-plethory_coassociativity}.

\begin{lem}
\label{prop:monoid-equations}
    The plethysm multiplication and unit for $\Lambda_+$ satisfy the monoid equations. 
\end{lem}

\begin{proof}
    The 2-plethysm multiplication for $\ksbar$ comes from the functor 
    \[
    \eta \maps U(\ksbar) \to \TRig(\ksbar, \ksbar) 
    \]
    given on objects by $\rho \mapsto (\tau \mapsto \tau \comp \rho)$. Its transpose $(\rho, \tau) \mapsto \tau \comp \rho$ is associative up to isomorphism by \cref{cor:pleth}. The usual decategorification procedure (interpret this map in $\CCat\ho$, apply $\CCat\ho(1, -)$, and compose with the action of $J$ on homs) leads to the plethysm structure on $\Lambda_+$
    \[ 
    U\Lambda_+ = H\ksbar \to \TRig\ho(\ksbar, \ksbar) \to \Rig(J(\ksbar), J(\ksbar)) = \Rig(\Lambda_+, \Lambda_+) 
    \]
    which is the following well-defined map on isomorphism classes: 
    \[ 
    [\rho] \mapsto ([\tau] \mapsto [\tau \comp \rho])
    \]
    Thus, the binary operation $([\rho], [\tau]) \mapsto [\tau \comp \rho]$ is associative on the nose. The unit equations also follow by applying  \cref{cor:pleth} and passing to isomorphism classes. 
\end{proof}

Thanks to the fourth perspective on $M$-plethories developed in Lemmas \ref{lem:M-plethory_h}--\ref{lem:M-plethory_coassociativity}, by proving Lemmas \ref{lem:h-is-rig-map}--\ref{prop:monoid-equations} we have completed the proof of the following theorem: 

\begin{thm}
\label{thm:lambdaplus-forms-rig-plethory}
    The 2-plethory structure on $\ksbar$ induces a rig-plethory structure on $\Lambda_+ = J(\ksbar)$, the rig of positive symmetric functions. 
\end{thm} 

\section{The ring-plethory of symmetric functions}
\label{sec:lambda}

The preceding development shows that the rig $\Lambda_+ = J(\ksbar)$ carries a canonical rig-plethory structure, giving a right adjoint comonad on $\Rig$ whose underlying functor
\[\Phi_{\Lambda_+} \maps \Rig \to \Rig\]
lifts $\hom(\Lambda_+, -) \maps \Rig \to \Set$ through the underlying-set functor $U \maps \Rig \to \Set$. What we now show is that the actual Grothendieck ring 
\[K(\ksbar) = \Z \otimes_\N J(\ksbar), \]
also denoted $\Lambda = \Z \otimes_\N \Lambda_+$, similarly carries the structure of a ring-plethory, more commonly known simply as a plethory. 

By way of background, group completion is the functor $\CMon \to \Ab$ left adjoint to the full inclusion functor $\Ab \hookrightarrow \CMon$. In terms of the symmetric monoidal product $\otimes_\N$ on $\CMon$, it is the functor that sends a commutative monoid $A$ to $\Z \otimes_\N A$.

\begin{lem}
The functor $\Z\otimes_\N - \maps (\CMon, \otimes_\N) \to (\Ab\Grp, \otimes)$ is symmetric strong monoidal.
\end{lem}

\begin{proof}  
    Just as in the more familiar case of ring modules, extension of scalars from modules of a rig to modules of a larger rig is a symmetric strong monoidal functor:
    \begin{align*} 
        (\Z \otimes_\N A) \otimes_\Z (\Z \otimes_\N B) 
        &\cong ((\Z \otimes_\N A) \otimes_\Z \Z) \otimes_\N B 
        \\&\cong (\Z \otimes_\N A) \otimes_\N B 
        \\&\cong  \Z \otimes_\N (A \otimes_\N B). \qedhere
    \end{align*}
\end{proof}
Being a (strong) symmetric monoidal functor, group completion takes commutative $\otimes_\N$-monoids in $\CMon$, which are commutative rigs, to commutative $\otimes$-monoids in $\Ab$ which are commutative rings. In fact, the functor $\Z \otimes_\N - \maps \Rig \to \Ring$ is left adjoint to the full embedding $i \maps \Ring \hookrightarrow \Rig$. It follows that 
\[
\Ring \xhookrightarrow{i} \Rig \xrightarrow{\Phi_{\Lambda_+}} \Rig
\]
is the lift of the representable functor $\Ring(\Lambda, -) \maps \Ring \to \Set$ through $U \maps \Rig \to \Set$, because 
\begin{align*} 
    U \circ \Phi_{\Lambda_+} \circ i 
    &\cong \Rig(\Lambda_+, -) \circ i 
    = \Rig(\Lambda_+, i-) 
    \\&\cong \Ring(\Z \otimes_\N \Lambda_+, -) 
    = \Ring(\Lambda, -).
\end{align*}

This lift gives the $\Set$-valued representable $\Ring(\Lambda, -)$ a rig structure $\Phi_{\Lambda_+} \circ i$, and our first task is to see that this extends to a ring structure, thus making $\Lambda$ a biring. Such a biring extension is unique up to unique isomorphism when it exists, because the embedding $i \maps \Ring \hookrightarrow \Rig$ is full and faithful. Put differently, a rig structure can be a ring structure in at most one way; we are simply asking ``does this rig structure on $\Ring(\Lambda, -)$ have the \emph{property} of being a ring structure?"

To show it does, it is necessary and sufficient to identify a suitable co-negation on $\Lambda$, where `negation' is in the sense of `additive inverse'. A co-negation is a ring map 
\[\nu \maps \Lambda \to \Lambda\] 
such that the following diagram in the category of rings commutes.
\begin{equation}
\label{eq:conegation}
\begin{tikzcd}
    &
    \Lambda \otimes \Lambda
    \arrow[rr, "1 \otimes \nu"]
    &&
    \Lambda \otimes \Lambda
    \arrow[dr, "\nabla_\Lambda"]
    \\
    \Lambda
    \arrow[ur, "\alpha_\Lambda"]
    \arrow[dr, swap, "\alpha_\Lambda"]
    \arrow[rr, "o_\Lambda"]
    &&
    \Z
    \arrow[rr, "!"]
    &&
    \Lambda
    \\
    &
    \Lambda \otimes \Lambda
    \arrow[rr, swap, "\nu \otimes 1"]
    &&
    \Lambda \otimes \Lambda
    \arrow[ur, swap, "\nabla_\Lambda"]
\end{tikzcd}
\end{equation}
Again, such a co-negation is unique if it exists. By cocommutativity of coaddition $\alpha_\Lambda$, commutativity of either pentagon implies commutativity of the other. Thus we focus on the top pentagon, which we call the \define{co-negation equation}. 

Before constructing the co-negation $\nu$, we review the Grothendieck ring $G(R)= \Z \otimes_{\mathbb{N}} R$ of a rig $R$, with a view toward categorification. First, $G(R)$ is a quotient rig of a rig $R[x]/(x^2 = 1)$. This rig is the same as the group rig $R[\Z_2]$, whose underlying additive monoid is $R \times R$ (regarding elements $a + bx$ modulo $(x^2-1)$ as ordered pairs $(a, b)$), and in which elements are multiplied by the rule 
\[(a, b) \cdot (a', b') = (aa' + bb', ab' + a'b).\]
Then, to form $G(R)$ as a quotient of $R[\Z_2]$, $q \maps R[\Z_2] \to G(R)$, one introduces a rig-congruence relation on $R[\Z_2]$, generated by a symmetric transitive relation $\sim$ defined by 
\[(a, b) \sim (c, d)\;\;\; \Leftrightarrow \;\;\; a+d = b + c\]
and we denote the $\sim$-equivalence class of $(a, b)$ by $a - b$. If the additive monoid of $R$ is cancellative, then $\sim$ is already transitive. Note that $\Lambda_+ = J(\ksbar)$ is cancellative because it is a free $\N$-module, by Maschke's theorem. 

The categorified analog of $\Lambda_+[\Z_2]$ is the 2-rig of $\Z_2$-graded Schur functors, which we denote as $\G$. The underlying category of $\G$ is the product $\ksbar \times \ksbar$, whose objects we write as $(C_0, C_1)$. This category can be equivalently described as the category of $\Z_2$-graded polynomial species, i.e.\ linearly enriched functors valued in $\Z_2$-graded finite-dimensional vector spaces, 
\[
F \maps k\S \to \mathsf{Gr}_{\Z_2}(\Fin\Vect),
\]
for which all but finitely many values $F(n)$ are zero. 

The tensor product is the usual graded tensor, and exactly mirrors multiplication in the group rig $R[\Z_2]$: 
\[
    (C_0, C_1) \otimes (D_0, D_1) = ((C_0 \otimes D_0) \oplus (C_1 \otimes D_1), (C_0 \otimes D_1) \oplus (C_1 \otimes D_0)).
\]
The symmetry 
\[
    \sigma_{C,D} \maps (C_0, C_1) \otimes (D_0, D_1) \to (D_0, D_1) \otimes (C_0, C_1)
\]
is the standard one involving a sign convention. Namely, if $C = (C_0, C_1)$ and $D = (D_0, D_1)$ are two graded Schur objects, then 
\begin{align*}
    (\sigma_{C, D})_0 &= \sigma_{C_0, D_0} \oplus -\sigma_{C_1, D_1} \maps (C_0 \otimes D_0) \oplus (C_1 \otimes D_1) \to (D_0 \otimes C_0) \oplus (D_1 \otimes C_1); 
    \\
    (\sigma_{C, D})_1 &= \sigma_{C_0, D_1} \oplus \sigma_{C_1, D_0} \maps (C_0 \otimes D_1) \oplus (C_1 \otimes D_0) \to (D_1 \otimes C_0) \oplus (D_0 \otimes C_1).
\end{align*}

\begin{lem}
    We have an isomorphism of rigs $\Lambda_+[\Z_2] \cong J(\G)$.
\end{lem}

Let $q \maps J(\G) \cong \Lambda_+[\Z_2] \to \Lambda$ be the quotient map. It is given by $q([C_0], [C_1]) = [C_0] - [C_1]$ for an object $(C_0, C_1)$ of $\G$.

\begin{lem}
\label{prop:J-on-boxtimes-G}
    The canonical rig map $J(\G) \otimes J(\G)\to J(\G \boxtimes \G)$ is an isomorphism. 
\end{lem}

Let $x$ abbreviate the generating object $\S(-, 1)$ (as a linearized representable) of $\ksbar$. Let $\phi_+ \maps \ksbar \to \G$ be the essentially unique 2-rig map that sends $x$ to the graded object $(x, 0) \in \G$. Let $\phi_- \maps \ksbar \to \G$ be the essentially unique 2-rig map that sends $x$ to the graded object $(0, x) \in \G$. Applying $J$, we obtain rig maps 
\[
    J(\phi_+) \maps J(\ksbar) \to J(\G), \qquad J(\phi_{-}) \maps J(\ksbar) \to J(\G).
\]

\begin{lem}
\label{prop:phi-plus-gives-id}
    The unique ring map $\Lambda \to \Lambda$ that extends the composite 
    \[
    \Lambda_+ = J(\ksbar) \xrightarrow{J(\phi_+)} J(\G) \xrightarrow{q} \Lambda
    \]
    of rig maps is the identity map on $\Lambda$. 
\end{lem} 

\begin{proof}
    It suffices to check that the composite is the inclusion $\Lambda_+ \hookrightarrow \Lambda$. Since $(x, 0)^{\otimes n} = (x^{\otimes n}, 0)$, it is clear that $\phi_+(\rho) = (\rho, 0)$ for any Schur object $\rho$. Hence $J(\phi_+)$ takes a class $[\rho]$ to $([\rho],[0])$, and $q$ takes the latter to $[\rho] - [0] = [\rho]$. This completes the proof. 
\end{proof}

\begin{defn}
    The \define{co-negation} on $\Lambda$ is the unique ring map $\nu \maps\Lambda \to \Lambda$ that extends the composite of rig maps
    \[
    \Lambda_+ = J(\ksbar) \xrightarrow{J(\phi_-)} J(\G) \xrightarrow{q} \Lambda.
    \]
\end{defn}

\begin{prop}
\label{thm:conegation}
    The map $\nu$ satisfies the co-negation equation, \cref{eq:conegation}. 
\end{prop}

The proof will be broken down into a series of simple lemmas. The first two are preparatory lemmas requiring no essentially new ideas.

\begin{lem}
\label{lem:reduction-one}
    The restriction of the composite 
    \[
    \Lambda \xrightarrow{\alpha_\Lambda} \Lambda \otimes \Lambda \xrightarrow{1 \otimes \nu} \Lambda \otimes \Lambda \xrightarrow{\nabla_\Lambda} \Lambda
    \]
    along the inclusion $\Lambda_+ \hookrightarrow \Lambda$ equals the composite 
    \[
    \Lambda_+ = J(\ksbar) \xrightarrow{J(\alpha)} J(\ksbar \boxtimes \ksbar) \xrightarrow{J(\phi_+ \boxtimes \phi_{-})} J(\G \boxtimes \G) \xrightarrow{J(\nabla_\G)} J(\G) \xrightarrow{q} \Lambda.
    \]
\end{lem}
The proof amounts to a simple diagram chase together with 
\cref{prop:phi-plus-gives-id} and \cref{prop:J-on-boxtimes-G}. 

\begin{lem}
\label{lem:reduction-two}
    The composite 
    \[
        \ksbar \xrightarrow{\alpha} \ksbar \boxtimes \ksbar \xrightarrow{\phi_+ \boxtimes \phi_-} \G \boxtimes \G \xrightarrow{\nabla_\G} \G
    \]
    is $\phi_{(x, x)} \maps \ksbar \to G$, the essentially unique 2-rig map that takes $x$ to $(x, x)$. 
\end{lem}

\begin{proof}
    This amounts to the element chase 
    \[
    x \mapsto (x \boxtimes 1) \oplus (1 \boxtimes x) \mapsto ((x, 0) \boxtimes 1) \oplus (1 \boxtimes (0, x)) \mapsto (x, 0) \oplus (0, x) = (x, x). \qedhere
    \]
\end{proof}

To make further progress, we introduce another 2-rig. Let $\DG$ denote the 2-rig of differential $\Z_2$-graded Schur functors. The objects of $\DG$ are tuples $(C_0, C_1, d_0, d_1)$ where $C_i \in \ksbar$, and $d_0 \maps C_0 \to C_1$ and $d_1 \maps C_1 \to C_0$ are morphisms of Schur objects such that $d_1 d_0=0$ and $d_0 d_1 = 0$. Morphisms are pairs of morphisms between Schur objects that respect the differentials $d_i$. The tensor product coincides with the tensor product of the underlying objects in $\G$, equipped with differentials defined by the usual rule $\partial(c \otimes d) = \partial c \otimes d + (-1)^{\mathrm{deg} c} c \otimes \partial d$. To be more precise: 
\[
d_{C \otimes D, 0} =\left( \begin{array}{c}
     1_{C_0} \otimes d_{D, 0} + d_{C, 0} \otimes 1_{D_0}  \\
     d_{C, 1} \otimes 1_{D_1} - 1_{C_1} \otimes d_{D, 1} 
\end{array}\right) \maps (C_0 \otimes D_0) \oplus (C_1 \otimes D_1) \to (C_0 \otimes D_1) \oplus (C_1 \otimes D_0)
\]
\[
d_{C \otimes D, 1} = \left( \begin{array}{c}
     1_{C_0} \otimes d_{D, 1} + d_{C, 0} \otimes 1_{D_1}  \\
     d_{C, 1} \otimes 1_{D_0} - 1_{C_1} \otimes d_{D, 0} 
\end{array}\right) \maps (C_0 \otimes D_1) \oplus (C_1 \otimes D_0) \to (C_0 \otimes D_0) \oplus (C_1 \otimes D_1).
\]
The symmetry isomorphism coincides with the symmetry on the underlying graded objects. 

The forgetful functor $U_{\DG} \maps \DG \to \G$ that forgets the differentials $d_i$ is manifestly a 2-rig map. Let $M_x \maps \ksbar \to \DG$ denote the 2-rig map that sends the generator $x$ to the complex where $d_0 = 1_x, d_1 = 0_x$, which we display as 
\[
\begin{tikzcd}
    x
    \arrow[r, bend left, "1"]
    &
    x.
    \arrow[l, bend left, "0"]
\end{tikzcd}\]
(This is the mapping cone of the identity on $(x, 0)$, hence the notation $M_x$.)  Clearly the following diagram commutes up to isomorphism:
\[
\begin{tikzcd}
    &
    \ksbar
    \arrow[dl, swap, "M_x"]
    \arrow[dr, "{\phi_{(x, x)}}"]
    \\
    \DG
    \arrow[rr, swap, "U_{\DG}"]
    &&
    \G.
\end{tikzcd}
\]
Combining this observation with \cref{lem:reduction-one} and \cref{lem:reduction-two}, we have the following result. 
\begin{lem}
\label{prop:reduction-three}
    The restriction of the composite (one side of the co-negation equation) 
    \[
    \Lambda \xrightarrow{\alpha_\Lambda} \Lambda \otimes \Lambda \xrightarrow{1 \otimes \nu} \Lambda \otimes \Lambda \xrightarrow{\nabla_\Lambda} \Lambda
    \]
    along the inclusion $\Lambda_+ \hookrightarrow \Lambda$ equals the composite 
    \[
    \Lambda_+ = J(\ksbar) \xrightarrow{J(M_x)} J(\DG) \xrightarrow{J(U_{\DG})} J(\G) \xrightarrow{q} \Lambda.
    \]
\end{lem}

We now turn to the other side  of the co-negation equation. This will be contextualized using the homology functor 
\[
    H \maps \DG \to \G
\]
which sends an object $(d_0 \maps C_0 \to C_1, d_1 \maps C_1 \to C_0)$ to its homology $(H_0, H_1)$. This functor is a 2-rig map, thanks to the algebra underlying the K\"unneth theorem.

\begin{lem}
\label{lem:reduction-four}
    The restriction of the composite (one side of the co-negation equation) 
    \[
    \Lambda \xrightarrow{o_\Lambda} \Z \xrightarrow{!} \Lambda
    \]
    along the inclusion $\Lambda_+ \hookrightarrow \Lambda$ equals the composite 
    \[
    \Lambda_+ = J(\ksbar) \xrightarrow{J(M_x)} J(\DG) \xrightarrow{J(H)} J(\G) \xrightarrow{q} \Lambda.
    \]
\end{lem}

\begin{proof}
    Because 
    \[
    H\left(\begin{tikzcd}
        x
        \arrow[r, bend left, "1"]
        &
        x
        \arrow[l, bend left, "0"]
    \end{tikzcd}\right) \cong (0, 0)
    \]
    it is clear that the following diagram in $\TRig$ commutes up to isomorphism, since both $\TRig$ composites $\ksbar \to \G$ send the generator $x$ to $(0, 0)$:
    \[
    \begin{tikzcd}
    &
    \Fin\Vect
    \arrow[rr, "!"]
    &&
    \ksbar
    \arrow[dr, "\phi_+"]
    \\
    \ksbar
    \arrow[ur, "\phi_0"]
    \arrow[rr, swap, "M_x"]
    &&
    \DG
    \arrow[rr, swap, "H"]
    &&
    \G. 
    \end{tikzcd}
    \]
    Applying $J$ to this last diagram, and augmenting to this a commutative triangle that expresses the equation $q \circ J(\phi_+) = i \maps \Lambda_+ \hookrightarrow \Lambda$ (\cref{prop:phi-plus-gives-id}), we arrive at a commutative diagram 
    \[
    \begin{tikzcd}
        &
        \mathbb{N}
        \arrow[rr, "!"]
        &&
        \Lambda_+ 
        \arrow[rr, "i"] 
        \arrow[dr, swap, "J(\phi_+)"]
        && 
        \Lambda
        \\
        \Lambda_+ = J(\ksbar)
        \arrow[ur, "o_{\Lambda_+}"]
        \arrow[rr, swap, "J(M_x)"]
        &&
        J(\DG)
        \arrow[rr, swap, "J(H)"]
        &&
        J(\G) \arrow[ur, swap, "q"]
    \end{tikzcd}
    \]
    and the result follows. 
\end{proof}

Combining \cref{prop:reduction-three} and \cref{lem:reduction-four}, the proof of the co-negation equation has been reduced to the following result.

\begin{lem}
    The following diagram commutes. 
    \[
    \begin{tikzcd}
        J(\ksbar) 
        \arrow[swap, d, "J(M_x)"] 
        \arrow[r, "J(M_x)"] 
        &
        J(\DG) 
        \arrow[r, "J(U_{\DG})"] 
        & 
        J(\G) 
        \arrow[d, "q"] 
        \\
        J(\DG) 
        \arrow[r, swap, "J(H)"] 
        & 
        J(\G) 
        \arrow[r, swap, "q"] 
        & 
        \Lambda
    \end{tikzcd}
    \]
\end{lem}

\begin{proof}
    It clearly suffices to show that $q$ coequalizes the two maps 
    \[
    J(U_{\DG}), J(H) \maps J(\DG) \rightrightarrows J(\G).
    \]
    But if $C$ is an object of $\DG$, with underlying graded object $(C_0, C_1)$ and homology object $(H_0, H_1)$, this says precisely
    \[
    [C_0] - [C_1] = [H_0] - [H_1].
    \]
    This is a well-known fact about the Euler characteristic of a chain complex---here a 2-term chain complex in $\ksbar$. It follows easily from our ability to split exact sequences in $\ksbar$, which gives these equations involving cycles $Z_i$ and boundaries $B_i$: 
    \[
    [C_0] = [Z_0] + [B_1], \quad [C_1] = [Z_1] + [B_0], \quad [Z_0] = [B_0] + [H_0], \quad 
    [Z_1] = [B_1] + [H_1]. \qedhere
    \]
\end{proof}

Having proved the co-negation equation, we have established the following result: 

\begin{thm}
\label{thm:lambda_biring}
The birig structure on $J(\ksbar)$ extends to a biring structure on $K(\ksbar) = \Z \otimes_\N J(\ksbar)$. 
\end{thm}

We may explicitly calculate the effect of co-negation on the class of a Schur functor $\rho$, by evaluating $\rho$ at the object $(0, x)$ in the 2-rig $\G$:
\[
[\rho](0,x) = q \left[\sum_{n \ge 0} \rho(n) \otimes_{S_n} (0, x)^{\otimes n}\right].
\]
The $n$th summand here lives in grade $1$ when $n$ is odd, and grade $0$ when $n$ is even. Taking into account that each transposition of tensor factors $(0, x)$ introduces a sign change, since $(0, x)$ is in odd degree, the result is 
\[
\nu([\rho]) = \sum_{n \ge 0} (-1)^n[\rho(n)][\det(n)]
\]
where $\det(n)$ is the alternating representation of $S_n$. Note that because $\rho(n)$ vanishes except for finitely many $n$, this is effectively a finite sum.

It is interesting to compare Joyal's `rule of signs' for so-called virtual species, i.e.\ formal differences of ordinary $\Set$-valued species \cite{AnalyticFunctors}. Joyal constructs a virtual species $\exp(-X)$ as a geometric series 
\[
\exp(X)^{-1} = \sum_{n \geq 0} (-1)^k (\exp(X)-1)^k = E_0(X) - E_1(X).
\]
Here
\[   E_i(X) = \sum_{n \ge 0} E_i[n] \times_{S_n} X^n \]
where the $n^{th}$ coefficient object $E_0[n]$ is the set of ordered partitions of the set $\{1, \ldots, n\}$ into an even number of blocks, while $E_1[n]$ is the set of ordered partitions into an odd number of blocks. He uses this to construct a virtual species  $F(-X)$ for any given species $F$, given by
\[
F(-X)[n] = F[n] \times \exp(-X)[n].
\]
Thus, Joyal's virtual species $\exp(-X)$ is analogous to the linear species whose $n$th coefficient object is $(-1)^n \det(n)$; indeed it is virtually equivalent to this upon applying linearization $\Set \to \Vect$ to the coefficient objects.

To conclude, we address the ring-plethory structure on $\Lambda$.  First, the rig-plethory structure on $\Lambda_+ = J(\ksbar)$, given by a comonad structure on $\Phi_{\Lambda_+} \maps \Rig \to \Rig$, \emph{restricts} to a comonad structure on $\Phi_\Lambda \maps \Ring \to \Ring$, making $\Lambda = K(\ksbar)$ into a ring-plethory. This follows easily from the full faithfulness of the inclusion $i \maps \Ring \to \Rig$. For example, from the comonad comultiplication $\delta \maps \Phi_{\Lambda_+} \to \Phi_{\Lambda_+} \circ \Phi_{\Lambda_+}$, we obtain a composite 
\[
    i \Phi_\Lambda \cong \Phi_{\Lambda_+} i \xrightarrow{\delta \circ i} \Phi_{\Lambda_+} \Phi_{\Lambda_+} i \cong \Phi_{\Lambda_+} i \Phi_\Lambda \cong i \Phi_\Lambda \Phi_\Lambda
\] 
and this composite gives a comultiplication $\Phi_\Lambda \to \Phi_\Lambda \Phi_\Lambda$, by full faithfulness of $i$. The requisite equations for a comonad are easily established using the naturality of the isomorphism $i \Phi_\Lambda \cong \Phi_{\Lambda_+} i$.   In summary, we have:

\begin{thm}
\label{thm:lambda_plethory}
    The biring structure on $\Lambda = K(\ksbar)$ given in \cref{thm:lambda_biring} extends to a ring-plethory structure. 
\end{thm} 

\subsection{Conclusions}

We have worked out the theory of categorified plethysm and used it to show that $\Lambda$ is a ring-plethory without any reference to the usual theory of symmetric functions.  While this was precisely our goal, the reader may still wonder: why does this ring-plethory structure on $\Lambda$ match the usual ring-plethory structure on the ring of symmetric functions?  And: what does all this have to do with $\lambda$-rings as they are usually defined?

Macdonald's book fills in some of the missing links \cite[Appendix IA]{Macdonald}.   Recall that we have established equivalences
\[   \Schur \simeq \Poly \simeq U(\ksbar) \]
For each $n \ge 1$ the exterior power operation $\Lambda^n \in \Schur$ thus gives an element $\lambda^n \in K(\ksbar)$.  In fact $\Lambda$ is the free ring on these elements.  Indeed, Macdonald gives an explicit isomorphism between $K(\ksbar)$ and the ring of symmetric functions, which maps $\lambda^n$ to the $n$th elementary symmetric function: the sum in $\Z[[x_1, x_2, \dots ]]$ of all products of $n$ distinct variables.  Macdonald also defines a plethysm product on symmetric functions, and shows that it is a decategorified version of the substitution product described in \cref{cor:pleth}.

Traditionally a $\lambda$-ring is defined as a ring equipped with operations $\lambda^n$ obeying identities motivated by properties of the exterior power operations $\Lambda^n \in \Schur$.  Our work gives another outlook on $\lambda$-rings.  According to the discussion following \cref{def:Mplethory}, we may view the ring-plethory structure on $\Lambda$ as the left adjoint monad structure on the functor $\Psi_\Lambda \maps \Ring \to \Ring$. The algebras of this monad were studied by Tall and Wraith \cite{TallWraith}, who showed that they are equivalent to `special $\lambda$-rings', which are nowadays usually just called $\lambda$-rings. In short, $\Psi_\Lambda$ is the monad for $\lambda$-rings. 

Furthermore, according to the discussion after \cref{def:Mplethory}, the left adjoint monad $\Psi_\Lambda$ gives a right adjoint comonad $\Phi_\Lambda \maps \Ring \to \Ring$ such that
\[\Phi_\Lambda(R) = \Ring(\Lambda, R). \]
Hazewinkel \cite[Sec.\ 16.59]{Hazewinkel} discusses this comonad and shows that $\Phi_\Lambda(R)$ is the so-called `big Witt ring' of $R$, a familiar object in the theory of $\lambda$-rings.  We may thus call $\Phi_\Lambda$ the \define{big Witt comonad}.  Since $\Lambda$ is the free ring on the elements $\lambda^n$, there is a bijection
\[  \begin{array}{ccl}
\Ring(\Lambda,R) & \to & 1 + t R[[t]] \subseteq R[[t]] \\ \\
f & \mapsto & 1 + \displaystyle{\sum_{n \ge 1} f(\lambda^n) t^n }
\end{array}
\]
Under this bijection addition in the big Witt ring corresponds to multiplication of formal power series with constant term $1$, while multiplication is given by a less obvious formula, explained for example by Lenstra \cite{Lenstra}.  See also Borger and Wieland for a more abstract treatment using ring-plethories \cite[Sec.\ 3.2]{Plethystic}.  

Now, Eilenberg and Moore \cite{EMAdjoint} showed that the category of algebras of any left adjoint monad is equivalent to the category of coalgebras of its associated right adjoint comonad. Thus, $\lambda$-rings are equivalent to coalgebras of the big Witt comonad.   Concretely, putting a $\lambda$-ring structure on $R$ is equivalent to a making it a coalgebra of the big Witt comonad with coaction
\[  \begin{array}{rcl}
\eta \maps R &\to&  1 + t R[[t]] \cong \Ring(\Lambda,R) \\ \\
r & \mapsto &\displaystyle{1 + \sum_{n \ge 1} \lambda^n(r) t^n}.
\end{array}
\]

A final consequence of our work is that for any 2-rig $\R$, the ring $K(\R)$ acquires the structure of a $\lambda$-ring.  To obtain this, we first put a $\Phi_{\Lambda_+}$-coalgebra structure on $J(\R)$ given by a rig map 
\[h \maps J(\R) \to \Phi_{\Lambda_+}(J(\R))\]
whose underlying function is given as a composite
\[
    UJ(\R) \cong \TRig\ho(\ksbar, \R) \to \Rig(J(\ksbar), J(\R)) = \Rig(\Lambda_+, J(\R)) \cong U\Phi_{\Lambda_+}J(\R).
\]
Then $K(\R)$ acquires a $\Phi_\Lambda$-coalgebra structure given by the unique ring map $K(\R) \to \Phi_\Lambda K(\R)$ that extends the rig map formed as the composite 
\[
    J(\R) \xrightarrow{\;h\;} \Phi_{\Lambda_+} J(\R) \xrightarrow{\Phi_{\Lambda_+} \iota} \Phi_{\Lambda_+} iK(\R) \cong i\Phi_\Lambda K(\R)
\]
where $\iota \maps J(\R) \to iK(\R)$ is the canonical inclusion.  In summary:

\begin{cor}
    For any 2-rig $\R$, $K(\R)$ naturally has the structure of a $\lambda$-ring.
\end{cor}

\bibliographystyle{alpha}
\bibliography{references}

\end{document}